\documentclass{article}
\usepackage[utf8]{inputenc}
\usepackage[english]{babel}
\usepackage{amsmath}
\usepackage{amsfonts}
\usepackage{amssymb}
\usepackage{amsthm}
\usepackage{color}
\usepackage{algpseudocode}
\usepackage{algorithm}

\usepackage{mathtools}

\usepackage{enumitem}

\usepackage{csquotes}

\usepackage[backend=biber, style=alphabetic,
]{biblatex}
\usepackage{graphicx}
\usepackage{lmodern}
\usepackage[left=2cm,right=2cm,top=2cm,bottom=2cm]{geometry}

\usepackage{amsthm}
\usepackage{aliascnt}
\usepackage{hyperref,etoolbox}
\usepackage[capitalize]{cleveref}

\hypersetup{
    hidelinks,
    pdftitle={On the Universality of Simple Trust-Region Algorithms},
    pdfauthor={Clemens Sirotenko}
}

\theoremstyle{plain}
\newtheorem{theorem}{Theorem}

\newaliascnt{lemma}{theorem}
\newtheorem{lemma}[lemma]{Lemma}
\aliascntresetthe{lemma}

\newaliascnt{proposition}{theorem}

\aliascntresetthe{proposition}

\newaliascnt{corollary}{theorem}
\newtheorem{corollary}[corollary]{Corollary}
\aliascntresetthe{corollary}

\theoremstyle{definition}
\newtheorem{assumption}{Assumption}

\theoremstyle{remark}
\newtheorem{remark}{Remark}

\Crefname{theorem}{Theorem}{Theorems}
\Crefname{lemma}{Lemma}{Lemmas}
\Crefname{proposition}{Proposition}{Propositions}
\Crefname{corollary}{Corollary}{Corollaries}
\Crefname{assumption}{Assumption}{Assumptions}
\Crefname{remark}{Remark}{Remarks}

\DeclareMathOperator*{\argmin}{arg\,min \;}

\newcommand*\Diff[1]{\mathop{}\!\mathrm{D}}

\usepackage{xcolor}
\usepackage{todonotes}

\newcounter{algvariant}[algorithm]
\renewcommand{\thealgvariant}{\Roman{algvariant}}

\crefname{algvariant}{variant}{variants}
\Crefname{algvariant}{Variant}{Variants}

\newcommand{\algvariantitem}[2]{%
  \refstepcounter{algvariant}%
  \item[\textbf{Variant~\thealgvariant: #1}]\label{#2}%
}

\newcommand{\Crefalgvariant}[2]{\Cref{#1}}

\usepackage{booktabs}
\usepackage{tabularx}
\usepackage{array}

\usepackage{threeparttable} 

\title{On the Universality of Simple Trust-Region Algorithms}

\author{
Clemens Sirotenko\thanks{Weierstrass Institute for Applied Analysis and Stochastics (WIAS), Berlin, Germany. Email: \texttt{sirotenko@wias-berlin.de}}
}

\begin{document}
\maketitle

\begin{abstract}
We establish universal complexity guarantees for quadratic trust-region methods and identify a common mechanism underlying their universal behavior under convexity, based on a function-gap model-decrease estimate that appears to be new in the trust-region literature. First, we prove that the basic trust-region method with inexact subproblem solves is universal under convexity. Under a $\nu$-H\"older-continuous Hessian, it attains the global complexity bound $\mathcal{O}(\varepsilon^{-1/(1+\nu)})$ for computing an $\varepsilon$-approximate minimizer, without knowledge of $\nu\in[0,1]$ or the corresponding H\"older constant. In the nonconvex
regime, the method retains the classical
$\mathcal{O}(\varepsilon^{-2})$ first-order complexity bound under the usual additional bounded-Hessian assumption. With suitably vanishing inexactness, it also recovers Q-superlinear local convergence for
$\nu=0$ and convergence of order $1+\nu$ for $\nu\in(0,1]$. Second, we show that the same convex-universal mechanism applies to a trust-region variant with exact subproblem solves and a simple modification of the acceptance ratio. This variant is universal simultaneously in the nonconvex, convex, and local regimes: it attains the optimal nonconvex
first-order complexity $\mathcal{O}(\varepsilon^{-(2+\nu)/(1+\nu)})$, while preserving the universal convex complexity and the local Newton rates. These guarantees
require no knowledge of $\nu$ or its H\"older constant. Both methods use the usual quadratic trust-region model and the classical
radius-update mechanism, without gradient-dependent radii or model modifications such as cubic, gradient, or tensor regularization.  The results show that the trust-region mechanism is inherently adaptive across nonconvex, convex, and locally strongly convex regimes, providing further theoretical support for the practical success of trust-region methods.
\end{abstract}

\section{Introduction}

We consider unconstrained optimization problems of the form
\begin{equation}
    \min_{x\in H} f(x) \tag{$\mathrm{P}$},
    \label{eq:main_problem}
\end{equation}
where $H$ is a finite-dimensional real Hilbert space; in particular, one may take $H=\mathbb{R}^d$, $d\geq 1$. We assume that $f$ is bounded below and that its set of minimizers is nonempty, and define
\begin{equation*}
    X_\star:=\arg\min_{x\in H} f(x),
    \qquad
    f_\star:=\min_{x\in H} f(x).
\end{equation*}
We focus on twice continuously differentiable functions $f:H\to\mathbb{R}$ with H\"older-continuous Hessians. The main goal of this paper is to derive universal global complexity guarantees in the nonconvex and convex
regimes, together with universal local convergence rates near nondegenerate minimizers, for simple trust-region algorithms \cite{conn2000trust}.

In this article, we call a method \emph{universal} in a given regime if a single algorithm, without using the H\"older exponent $\nu$ or constant $L_\nu$ as algorithmic inputs, attains for every $\nu\in[0,1]$ the corresponding H\"older-dependent rate in that regime. Under convexity,
this is the global complexity bound
\[
    \mathcal{O}\bigl(\varepsilon^{-1/(1+\nu)}\bigr)
\]
for computing an $\varepsilon$-approximate minimizer. In the nonconvex regime, it is the optimal (see \cite{CartisGouldToint2018}) first-order complexity bound
\[
    \mathcal{O}\bigl(\varepsilon^{-(2+\nu)/(1+\nu)}\bigr).
\]
In the local regime, we speak of \emph{local universality} when the method automatically recovers the classical behavior of Newton's method: Q-superlinear convergence for $\nu=0$ and convergence of order $1+\nu$
for $\nu\in(0,1]$. Accordingly, we call a method \emph{universal across all regimes} if the same algorithm attains these H\"older-dependent guarantees in the nonconvex, convex, and local regimes without knowledge of $\nu$ or $L_\nu$.

The main message of this paper is that such universality is already largely inherent in simple trust-region mechanisms. We first show that the classical ratio-test trust-region method with inexact subproblem solves is universal under convexity and locally universal near nondegenerate minimizers. We then show that a simple, recently proposed modification of the acceptance ratio yields a trust-region method that is universal across all three regimes.

\paragraph{Related literature}

It is well known that trust-region methods have strong convergence properties in
the nonconvex regime, since their subproblems naturally accommodate indefinite
Hessian approximations and negative curvature. Complexity guarantees in this
setting for variants of the classical trust-region method have been available for
some time; see \cite{conn2000trust} for a classical account and
\cite{cartis2022evaluation} for a modern monograph. 

By contrast, the complexity
of trust-region methods under convexity has received considerably less attention.
The work \cite{grapiglia2016worst} analyzes the convex case and establishes the
standard $\mathcal{O}(\varepsilon^{-1})$ iteration complexity bound under
Lipschitz continuity of the gradient. Its analysis is related to
\cite{cartis2012evaluation}, where an adaptive cubic Newton-type method is studied
under convexity and the faster complexity bound
$\mathcal{O}(\varepsilon^{-1/2})$ is obtained under additional Lipschitz continuity of the
Hessian. The latter result is a major inspiration for this article. The recent
paper \cite{DiouaneHabiboullahOrban2026UnboundedTR}, although primarily concerned
with unbounded Hessian approximations, also recovers the
$\mathcal{O}(\varepsilon^{-1})$ convex complexity bound as a special case.

 The closest direct point of comparison to the present article is
\cite{jiang2026beyond}. The authors state that their work is the first to prove a
global iteration complexity bound of $\mathcal{O}(\varepsilon^{-1/2})$ under
convexity and Lipschitz continuity of the Hessian for a trust-region variant,
while also obtaining an almost optimal $\mathcal{O}(\varepsilon^{-3/2}\log(1/\varepsilon))$ iteration complexity for computing a \emph{second-order stationary} point and quadratic local convergence under local strong convexity. Their method uses exact subproblem solves and
combines a gradient-regularized quadratic model with a trust-region constraint
whose radius is scaled by the gradient norm. The update mechanism for the radius
and the gradient-regularization parameter differs substantially from the
classical ratio-test trust-region update. Moreover, the analysis is restricted
to Lipschitz-continuous Hessians and does not provide complexity guarantees for
the full H\"older scale.

The recent work \cite{jiang2025accelerating} builds on a similar trust-region mechanism to that in \cite{jiang2026beyond} and establishes accelerated complexity bounds under convexity and Lipschitz continuity of the Hessian, while preserving fast local convergence through a procedure that detects entry into the local convergence regime. However, this method requires knowledge of the problem constants.

Outside the trust-region literature, a number of regularized Newton methods have
been developed in recent years. This line goes back to
\cite{nesterov2006cubic}, where a cubic regularization term is used to obtain the
global complexity bound $\mathcal{O}(\varepsilon^{-1/2})$ under convexity and
Lipschitz continuity of the Hessian. Since the resulting cubic model may be more
difficult to handle in practice, alternative model functions based on gradient
regularization have been proposed; see, for example,
\cite{mishchenko2023regularized,doikov2024gradient}. In these methods, the
quadratic model is regularized by a term depending on the gradient norm, for
instance by replacing the Hessian with
\[
    \nabla^2 f(x_k)+c\|\nabla f(x_k)\|^\alpha I,
\]
for suitable parameters $c>0$ and $\alpha\geq 0$. These approaches also attain
the global complexity bound $\mathcal{O}(\varepsilon^{-1/2})$ under convexity. Related results
using gradient regularization have been developed for self-concordant functions
\cite{doikov2025minimizing}, distributed optimization \cite{zhang2024communication}, quasi-Newton approximations \cite{wang2024global}, and nonsmooth problems
\cite{alphonse2025leap,alphonse2026skip}.

In recent years, there has also been growing interest in universal Newton-type algorithms whose convergence guarantees adapt to the H\"older regularity of the Hessian without prior knowledge of the corresponding regularity parameters. An early universal method based on cubic-type regularization for functions with H\"older-continuous Hessians was presented in \cite{grapiglia2017regularized}. Subsequent developments include methods allowing inexact subproblem solves and using a trust-region-type update of the regularization parameter \cite{cartis2019universal}, as well as higher-order extensions \cite{grapiglia2020tensor}.
Universality of gradient-regularized methods under
H\"older-continuous Hessians was investigated recently in
\cite{doikov2024super}. That work proves that a variant of the gradient-regularized Newton method with exact subproblem solves is universal
in the convex regime over the full scale $\nu\in[0,1]$. It also establishes superlinear local convergence, although with a local order slightly weaker than that of the classical Newton method; see \Cref{tab:method_comparison}. Nonconvex complexity guarantees are not
provided there. More recently, gradient-regularized Newton methods have also been studied in the nonconvex regime; see
\cite{zhou2026regularized,zeng2026adaptive}. These works establish optimal or H\"older-adaptive nonconvex complexity guarantees using, among other ingredients, the subproblem solver proposed in \cite{royer2020newton}. Corresponding universal convex complexity guarantees are not established in these works.

The universal methods discussed in the preceding paragraphs are regularized Newton methods. As a first step toward universal trust-region methods, we mention the contracting-domain and contracting-point Newton methods \cite{doikov2020convex,DoikovNesterov2023AffineInvariantContractingPoint}. These methods attain the universal convex rate
$\mathcal{O}(\varepsilon^{-1/(1+\nu)})$ under a $\nu$-H\"older-continuous Hessian. In special cases, their subproblems admit a trust-region interpretation. However,  unlike the classical ratio-test trust-region method considered here, these schemes use prescribed contraction coefficients rather than acceptance ratios and adaptive radius updates. Moreover, universal nonconvex complexity guarantees and the corresponding locally universal Newton rates have not been established for these methods.

\begin{table}[t!]
\centering
\footnotesize
\setlength{\tabcolsep}{4pt}
\renewcommand{\arraystretch}{1.25}
\begin{threeparttable}
\begin{tabularx}{\textwidth}{@{}
>{\raggedright\arraybackslash}p{0.24\textwidth}
>{\centering\arraybackslash}p{0.145\textwidth}
>{\centering\arraybackslash}p{0.145\textwidth}
>{\centering\arraybackslash}p{0.16\textwidth}
>{\raggedright\arraybackslash}X
@{}}
\toprule
\textbf{Method}
&
\textbf{Convex}
&
\textbf{Nonconvex }
&
\textbf{Local \tnote{1}}
&
\textbf{Algorithmic features}
\\
\midrule
\addlinespace[0.7em]
\textbf{Regularized Newton Method}
\newline
{\scriptsize Grapiglia--Nesterov~\cite{grapiglia2017regularized}}
&
\(\mathcal{O}\!\left(\varepsilon^{-\frac{1}{1+\nu}}\right)\)
&
\(\mathcal{O}(\varepsilon^{-(2 + \nu)/(1+\nu)})\)
&
{\scriptsize not established}
&
Adaptive cubic Newton method with backtracking search over regularization parameter; not a trust-region method
\\
\addlinespace[0.7em]

\textbf{Adaptive cubic regularization (ARC)}
\newline
{\scriptsize Cartis--Gould--Toint~\cite{cartis2011adaptive2,cartis2012evaluation}}
&
\(\mathcal{O}(\varepsilon^{-1/2})\)
\newline
{\scriptsize for \(\nu=1\)}
&
\(\mathcal{O}(\varepsilon^{-3/2})\)
\newline
{\scriptsize for \(\nu=1\)}
&
Quadratic
\newline
{\scriptsize for \(\nu=1\)}
&
Adaptive cubic Newton method with ratio test and adaptive
regularization parameter; not a trust-region method
\\
\addlinespace[0.7em]

\textbf{Super-universal Newton}
\newline
{\scriptsize Doikov et al.~\cite{doikov2024super}}
&
\(\mathcal{O}\!\left(\varepsilon^{-\frac{1}{1+\nu}}\right)\)
&

{\scriptsize not established}
&
Superlinear
\newline
{\scriptsize order \(1+\frac{\nu}{1+\nu}\)}
&
Gradient-regularized Newton method; not a trust-region method
\\
\addlinespace[0.7em]

\textbf{Universal trust-region}
\newline
{\scriptsize Jiang et al.~\cite{jiang2026beyond}}
&
\(\mathcal{O}(\varepsilon^{-1/2})\)
\newline
{\scriptsize for \(\nu=1\)}
&
$\mathcal{O}(\varepsilon^{-3/2}\log(1/\varepsilon))$\,\tnote{2} 
\newline
{\scriptsize for \(\nu=1\)}
&
Quadratic
\newline
{\scriptsize for \(\nu=1\)}
&
Exact trust-region method with weighted radius, regularization, and
nonstandard radius updates
\\
\addlinespace[0.7em]

\textbf{Contracting trust-region}
\newline
{\scriptsize Doikov, Nesterov%
~\cite{nesterov2018complexity,doikov2020convex}}
&
\(\mathcal{O}\!\left(\varepsilon^{-\frac{1}{1+\nu}}\right)\)
&
{\scriptsize not established}
&
Linear\tnote{3}
&
Contracting-domain Newton scheme with prescribed contraction
coefficient; admits a trust-region interpretation for indicator constraints
\\
\addlinespace[0.7em]
\midrule
\addlinespace[0.7em]
\textbf{Inexact basic trust-region}
\newline
{\scriptsize This work and \cite{conn2000trust}}
&
\(\mathcal{O}\!\left(\varepsilon^{-\frac{1}{1+\nu}}\right)\)
&
\(\mathcal{O}(\varepsilon^{-2})\) \tnote{4}
&
Superlinear
\newline
{\scriptsize order \(1+\nu\)}
&
Classical trust-region framework with inexact subproblem solves
\\
\addlinespace[0.7em]
\textbf{Exact UniCAT method}
\newline
{\scriptsize This work and \cite{hamad2024simple}}
&
\(\mathcal{O}\!\left(\varepsilon^{-\frac{1}{1+\nu}}\right)\)
&
\(\mathcal{O}(\varepsilon^{-(2 + \nu)/(1+\nu)})\)
&
Superlinear
\newline
{\scriptsize order \(1+\nu\)}
&
Trust-region framework with slightly modified acceptance ratio test and exact subproblem solves.\\
\bottomrule
\end{tabularx}
\begin{tablenotes}
\footnotesize
\item[1] Local rates are stated under local strong convexity or nonsingularity
of the Hessian at the limit point.
\item[2] The authors in \cite{jiang2026beyond} prove convergence to a second-order stationary point; we consider only first-order stationarity.
\item[3] The work \cite{doikov2020convex} establishes global linear convergence under global strong convexity and knowledge of problem constants.
\item[4] This complexity bound is derived under additional uniform boundedness of the Hessian, see \cite[Theorem~3.2.1]{cartis2022evaluation} and \Cref{thm:tr-nonconvex-complexity}. 
\end{tablenotes}

\end{threeparttable}
\caption{Comparison of global and local complexity guarantees for recent
Newton-type and trust-region methods under Hölder-continuous Hessians.} 
\label{tab:method_comparison}
\end{table}

\paragraph{Literature gap}
Despite the long history and widespread use of trust-region methods, their complexity analysis under convexity, particularly under additional regularity
of the Hessian, remains underdeveloped. To the best of our knowledge, apart from \cite{jiang2026beyond} and the follow-up work \cite{jiang2025accelerating}, no trust-region method has been shown to attain the Hessian-based convex complexity bound $\mathcal{O}(\varepsilon^{-1/2})$. In particular, no classical ratio-test
trust-region method has been shown to attain the universal convex complexity $\mathcal{O}(\varepsilon^{-1/(1+\nu)})$ over the full H\"older scale. More
broadly, no simple trust-region method based on the usual quadratic model and
an adaptive acceptance-ratio and radius-update mechanism has been shown to be
universal simultaneously in the nonconvex, convex, and local regimes.

\paragraph{Contributions}
This paper closes the gap above and identifies universal behavior in simple trust-region methods across the nonconvex, convex, and local regimes.

First, we prove that the basic inexact trust-region method with the standard
ratio test is universal under convexity. Specifically, for every $\nu\in[0,1]$, the method attains the global iteration complexity
\begin{equation}
    \mathcal{O}\bigl(\varepsilon^{-1/(1+\nu)}\bigr)
\end{equation}
for computing an $\varepsilon$-approximate minimizer under a
$\nu$-H\"older-continuous Hessian. The algorithm does not require knowledge of $\nu$ or $L_\nu$, and no $\nu$-dependent parameter tuning is used. Under
the usual additional bounded-Hessian assumption, the method retains the classical $\mathcal{O}(\varepsilon^{-2})$ complexity for computing an
$\varepsilon$-first-order stationary point in the nonconvex regime. Moreover, under a suitably vanishing inexactness condition, it recovers the
locally universal Newton behavior: Q-superlinear convergence for $\nu=0$ and convergence of order $1+\nu$ for $\nu\in(0,1]$.

In contrast to \cite{jiang2026beyond}, our convex result applies directly to the basic ratio-test trust-region framework with the usual quadratic model and allows inexact subproblem solves. The method in
\cite{jiang2026beyond} uses exact solutions of gradient-regularized trust-region subproblems, together with gradient-dependent radii and a nonstandard update mechanism. Moreover, while \cite{jiang2026beyond} treats the Lipschitz-Hessian endpoint $\nu=1$, our
analysis covers the complete scale $\nu\in[0,1]$.

Second, we prove that the simple modification of the acceptance ratio proposed in \cite{hamad2024simple} produces a trust-region method that is universal across all three regimes. In the nonconvex setting, the method
attains the optimal first-order iteration complexity
\begin{equation}
    \mathcal{O}\bigl(\varepsilon^{-(2+\nu)/(1+\nu)}\bigr).
\end{equation}
At the same time, it preserves the universal convex complexity $\mathcal{O}(\varepsilon^{-1/(1+\nu)})$ and the locally universal Newton rate. Thus, universality in the nonconvex regime is obtained by modifying
only the acceptance ratio, without changing the quadratic trust-region model or introducing gradient, cubic, or tensor regularization.

For this second variant, we assume exact trust-region subproblem solves. The work \cite{hamad2024simple} permits inexact subproblem solves and establishes the optimal nonconvex complexity at the Lipschitz endpoint
$\nu=1$. Our contribution is to extend the complexity analysis to the full H\"older scale $\nu\in[0,1]$ and, simultaneously, to establish universal
convex complexity and locally universal convergence. 

Thus, this second variant closes the three-regime universality gap identified above for a single trust-region algorithm.

An overview of the theoretical results established in this article is given in \Cref{tab:method_comparison}.

\paragraph{Limitations}
The present approach has some limitations.

First, our analysis of the UniCAT method, 
\Crefalgvariant{var:variant2}{} of \Cref{alg:abstract_tr_algorithm} relies
on exact subproblem solves. Although this assumption simplifies the
analysis and provides a clear view of the underlying mathematical ideas, a
more practical treatment incorporating inexact subproblem solves would be
highly desirable.

Second, in contrast to the super-universal Newton method proposed in
\cite{doikov2024super}, our framework is restricted to derivative
information of at most second order. Extending the framework to incorporate
third- or higher-order derivative information, while establishing the
corresponding adaptivity guarantees, would be an interesting direction for
future research.

Third, although \cite{jiang2026beyond} also relies on exact subproblem
solves, the authors establish nonconvex convergence guarantees for
second-order stationary points, whereas our analysis considers only
first-order stationarity. While our first-order complexity bounds avoid the
additional logarithmic factors appearing in \cite{jiang2026beyond}, an
extension of our analysis to second-order stationary points would also be
of interest.

\section{Algorithmic framework}\label{sec:algorithmic_framework}

We now introduce the parameterized trust-region framework studied in this
paper; see \Cref{alg:abstract_tr_algorithm}. It is based on the classical
basic trust-region method described, for instance, in
\cite{conn2000trust,cartis2022evaluation}, with an optional modification of the acceptance ratio proposed in \cite{hamad2022consistently,hamad2024simple}.

For a given iterate $x_k \in H$, $k \in \mathbb N$, we write
\begin{equation*}
  g_k := \nabla f(x_k),  \qquad H_k := \nabla^2 f(x_k),
\end{equation*}
and consider the quadratic model
\begin{equation}
 m_k(s) := f(x_k) + \langle g_k, s \rangle
 + \tfrac{1}{2} \langle H_k s,s\rangle .
 \label{eq:def_model_problem_k}
\end{equation}
Trust-region methods seek, at every iteration, approximate minimizers of
$m_k$ subject to $\|s\| \leq \Delta_k$, where $\Delta_k > 0$ is a
trust-region radius that is updated after every iteration. The classical
trust-region method accepts or rejects the trial step $s_k$ based on the ratio
\begin{equation}
         \rho_k :=
         \frac{f(x_k) - f(x_k+s_k)}
         {m_k(0)- m_k(s_k)} .
         \label{eq:standard_rho_def}
\end{equation}
Following the classical trust-region literature, at iteration $k$ we refer to $a_k(s)$ as the \emph{actual decrease} and to $p_k(s)$ as the \emph{predicted decrease}. These quantities are defined by
\begin{equation}
a_k(s) := f(x_k) - f(x_k+s), \qquad p_k(s) := m_k(0) - m_k(s). \label{eq:def_actual_decrease_predicted_decrease}
\end{equation}
Based on \cite{hamad2024simple}, we modify the ratio \eqref{eq:standard_rho_def} by adding the term
\[
\frac{\vartheta}{2}
\min \left \{ \|\nabla f(x_k)\|,\|\nabla f(x_k+s_k)\|\right\}\|s_k\|,
\qquad \vartheta \geq 0,
\]
to the denominator. This enforces a larger function decrease for large steps;
cf. \cite{hamad2024simple}. The resulting modified acceptance ratio is given in
\eqref{eq:rho_cat_def}.

The authors of \cite{hamad2024simple} prove that an inexact variant of their
algorithm, equipped with the modified ratio, achieves an iteration complexity of $\mathcal{O}(\varepsilon^{-3/2})$ under Lipschitz continuity of the Hessian for computing a first-order stationary point. This improves upon the classical trust-region algorithm, which is known to achieve only a worst-case complexity
of $\mathcal{O}(\varepsilon^{-2})$ under the same assumptions; cf. \cite{cartis2022evaluation}. While \cite{curtis2017trust} was the first work to achieve a complexity bound of $\mathcal{O}(\varepsilon^{-3/2})$ for a modified trust-region variant, the modification considered here and in \cite{hamad2022consistently} is a significant simplification compared to the approach in \cite{curtis2017trust} and improves the dependence on the Lipschitz constant $L_1>0$ of the Hessian from $L_1^{3/2}$ to $L_1^{1/2}$.

The resulting parameterized trust-region framework is stated in
\Cref{alg:abstract_tr_algorithm}. For $\vartheta=0$, it reduces to the
classical ratio-test trust-region method. For $\vartheta>0$, it uses the
modified acceptance ratio proposed in \cite{hamad2022consistently,hamad2024simple}; we reserve the name
\emph{UniCAT} for this modified-ratio specialization.

\begin{algorithm}[h!]
\caption{Parameterized quadratic trust-region framework} 
\label{alg:abstract_tr_algorithm}
\begin{algorithmic}[1]
\Require Initial point $x_0 \in H$, initial radius $\Delta_0 > 0$,
parameter $\vartheta\geq 0$,
\[
0<\eta_1\leq \eta_2<1,
\qquad
0<\gamma_1\leq \gamma_2<1 \leq \gamma_3.
\]
\For{$k = 0,1,2,\ldots$}
    \State If $\|g_k\|=0$, terminate and return $x_k$.
    \State Otherwise, compute a trust-region step $s_k \in H$, with $\|s_k\|\leq \Delta_k$ according to the chosen version, that is,
    \begin{align}
    s_k \approx
    \argmin_{\|s\|\leq \Delta_k} m_k(s),
    \label{eq:abstract_s_k_update}
    \end{align}
    \hspace{0.42cm} where the meaning of $\approx$ is specified in \Crefalgvariant{var:variant1}{} and \Crefalgvariant{var:variant2}{}.
\State Set $x_{k,+} := x_k + s_k$ and update the trust-region acceptance ratio
 \begin{equation}
\rho_k^\vartheta :=
\frac{f(x_k) - f(x_{k,+})}
{m_k(0)- m_k(s_k)
+ \frac{\vartheta}{2}
\min\{ \|\nabla f(x_k)\|,\|\nabla f(x_{k,+})\|\} \|s_k\|} .
\label{eq:rho_cat_def}
\end{equation}

    \State Update the iterate:
    \begin{equation}
        x_{k+1} =
        \begin{cases}
            x_k + s_k, & \text{if } \rho_k^\vartheta \geq \eta_1, \\[2mm]
            x_k,       & \text{otherwise.}
        \end{cases}
    \end{equation}

    \State Update the trust-region radius:
    \begin{equation}
    \Delta_{k+1} \in
   \begin{cases}
[\gamma_1 \Delta_k,\gamma_2 \Delta_k],
& \rho_k^\vartheta < \eta_1, \\[0.4em]
[\gamma_2 \Delta_k, \Delta_k],
& \eta_1 \leq \rho_k^\vartheta < \eta_2, \\[0.4em]
[\Delta_k,\gamma_3 \Delta_k],
& \rho_k^\vartheta \geq \eta_2 .
\end{cases}
\end{equation}

\EndFor
\end{algorithmic}
\end{algorithm}
\paragraph{Two variants of \Cref{alg:abstract_tr_algorithm}.}
We study two variants of \Cref{alg:abstract_tr_algorithm}. They differ in the
choice of the parameter $\vartheta \geq 0$ and consequently in the definition of the acceptance ratio
\eqref{eq:rho_cat_def} and in the accuracy with which the trust-region
subproblem \eqref{eq:abstract_s_k_update} is solved.

\begin{description}[leftmargin=0pt,labelindent=0pt,labelsep=0.75em]

\algvariantitem{Basic inexact trust-region method}{var:variant1}
In this version, we choose $\vartheta=0$. Hence $\rho_k^\vartheta$ reduces to
the standard ratio of actual to predicted reduction. Given a sequence of
inexactness parameters $\tau_k \in [0,\bar{\tau})$ and a given upper bound $0<\bar \tau$, the step $s_k$ is required to satisfy
\begin{align}
    \|s_k\| &\leq \Delta_k, \label{eq:eq_alg_feasibility}\\
    m_k(0)-m_k(s_k)
    &\geq
    \frac{1}{1+\tau_k}
    \bigl(
        m_k(0)-m_k(s_k^\star)
    \bigr),
    \qquad
    s_k^\star \in
    \argmin_{\|s\|\leq \Delta_k} m_k(s).
    \label{eq:fractional_model_descent}
\end{align}

\algvariantitem{Exact universal consistently adaptive trust-region (UniCAT) method}{var:variant2}
In this version, inspired by the modification proposed in
\cite{hamad2024simple}, we choose $\vartheta>0$ and solve the trust-region subproblem exactly, that is,
\begin{equation}
    s_k \in
    \argmin_{\|s\|\leq \Delta_k} m_k(s).
    \label{eq:exact_tr_subproblem}
\end{equation}
For the full nonconvex complexity analysis of this variant, we assume in addition that
\[
    0 < \eta_1 \leq \eta_2 < \frac{1}{1+\vartheta}.
\]
This condition can always be ensured by choosing the acceptance thresholds
sufficiently small. For the convex and local convergence results, the weaker condition
\[
    0 < \eta_1 \leq \eta_2 < \frac{1}{1+\vartheta/2}
\]
is sufficient.
\end{description}

\begin{remark}[On the inexactness of the algorithm]
The inexactness criterion for \Crefalgvariant{var:variant1}{} of
\Cref{alg:abstract_tr_algorithm}, namely \eqref{eq:fractional_model_descent},
may look somewhat abstract at first sight. However, it simply requires the computed step to achieve a fixed fraction of the optimal trust-region model reduction, and is closely related to standard requirements used in trust-region methods. For instance, the celebrated method of Mor\'e and Sorensen \cite{more1983computing} computes nearly exact solutions of the quadratic
trust-region subproblem and can be stopped once a suitable subproblem accuracy has been reached. Modern variants and implementations of this approach are also
available; see, for example, \cite{gould2010solving}. Since the precise choice of such a solver is not relevant for our analysis, we do not discuss these variants further here. We note, however, that \eqref{eq:fractional_model_descent} is stronger than,
for instance, the standard Cauchy-decrease condition
\eqref{eq:standard_cauchy_decrease}, which is commonly used as a first-order inexactness criterion and can be satisfied by fast CG-type methods such as the
Steihaug--Toint truncated conjugate-gradient method
\cite{toint1981sparsity,steihaug1983conjugate,conn2000trust}. 

Regarding \Crefalgvariant{var:variant2}{} of
\Cref{alg:abstract_tr_algorithm}, we note that \cite{hamad2024simple} presents a rather involved inexactness criterion combined with the same acceptance ratio \eqref{eq:rho_cat_def}. A simple inexactness condition of the form \eqref{eq:fractional_model_descent} does not seem sufficient to obtain similar guarantees. However, we believe that a similar inexactness framework as the one of \cite{hamad2024simple} can be combined with the present setting.
\end{remark}

\section{Assumptions and preliminaries}\label{sec:assumptions}
We collect several assumptions and preliminary facts that will be used throughout the analysis of \Cref{alg:abstract_tr_algorithm}.
\subsection{Standing assumptions}
Let us first state the assumptions necessary for our analysis.
\begin{assumption}[Nonconvex regime]\label{ass:nonconvex_assumptions} We assume that the following conditions hold:
\begin{itemize}
\item[A1)] The objective  $f:H\to\mathbb R$ is twice continuously differentiable, bounded below and with a nonempty set of minimizers $X_\star$,
\item[A2)] The Hessian $\nabla^2 f$ is globally $(L_\nu,\nu)$-Hölder continuous for some
$L_\nu>0$ and $\nu\in[0,1]$, that is, 
\begin{equation*}
\|\nabla^2 f(x)-\nabla^2 f(y)\|
\le L_\nu\|x-y\|^\nu
\qquad \forall x,y\in H.    
\end{equation*}
\end{itemize}
\end{assumption}
For \Cref{thm:tr-nonconvex-complexity} we additionally need that the Hessians are uniformly bounded.
\begin{assumption}[Boundedness of the Hessians]\label{ass:bounded_hessian}
We assume that
\begin{itemize}
    \item[A3)]  the Hessians are uniformly bounded, i.e. 
\begin{equation}
    M_H := \sup_{x \in H} \|\nabla^2 f(x) \|_{\mathrm{op}} < + \infty.  \label{eq:bounded_hessian}
\end{equation}
\end{itemize}
    
\end{assumption}
\noindent
If \Cref{ass:nonconvex_assumptions} holds with $\nu =0$ then \Cref{ass:bounded_hessian} holds automatically.

Also note that under \Cref{ass:bounded_hessian} by mean value theorem we may also deduce the global Lipschitz continuity of $\nabla f$:
\begin{equation}
    \| \nabla f(x) - \nabla f(y)\| \leq M_H \|x-y\| \qquad \text{for $x,y \in H.$} \label{eq:Lip_cont_grad_f}
\end{equation}
In addition to the nonconvex regime, we will also investigate iteration complexity if the function $f$ happens to be convex. For this setting we make the following assumptions:
\begin{assumption}[Convex regime]\label{ass:convex_assumptions}
Let \Cref{ass:nonconvex_assumptions} hold. In addition, assume that
\begin{itemize}
    \item[A4)] $f:H\to\mathbb R$ is convex.
    \item[A5)] The initial distance to the solution set satisfies
\begin{equation}
    D:=\sup\left\{ \operatorname{dist}(x,X_\star):x \in H  \text{ such that } f(x)\leq f(x_0)\right\}<\infty. \label{eq:assumption_sublevel}
\end{equation}
\end{itemize}
\end{assumption}
\noindent
Also we define the function gap: 
\begin{equation}
    F_k := f(x_k) - f_\star, \label{eq:def_function_gap}
\end{equation}
which will serve as the stationarity measure in the section on the convex regime.
\subsection{Model approximation and trust-region subproblems}
We start with a general approximation theorem which is classical in optimization
\begin{lemma}[Model approximation]\label{lem:model_approximation}
Let \Cref{ass:nonconvex_assumptions} hold. Also let $m_k$ denote the model function introduced in \eqref{eq:def_model_problem_k} and let $(L_\nu,\nu) \in (0,\infty) \times [0,1]$ be the global H\"older constant  of $\nabla^2 f$. Then one has the remainder estimates
\begin{align}
|f(x_k+s)-m_k(s)|
&\leq
\frac{L_{\nu}}{(1+\nu)(2+\nu)}\|s\|^{2+\nu}
\qquad &\forall s\in H, \label{eq:eq_model_approx_func}\\
\|\nabla f(x_k + s) - \nabla f(x_k) - \nabla^2f(x_k)[s]\| 
&\leq \frac{L_\nu}{1+\nu} \|s\|^{1+\nu} \qquad &\forall s\in H.\label{eq:eq_model_approx_grad}
\end{align}
\end{lemma}
\begin{proof}
    The proof is standard and follows directly from the second-order Taylor remainder; see \cite[Eq.~(2.7)]{grapiglia2017regularized}.
\end{proof}
For the subsequent analysis we will often need the following quantities
\begin{equation}
    C_f:=\frac{L_\nu}{(1+\nu)(2+\nu)},
    \qquad
    C_g:=\frac{L_\nu}{1+\nu}, \qquad \alpha   := \frac{1}{1 + \nu}.\label{eq:def_C_f_C_g_alpha}
\end{equation}

 We also need to characterize the solutions of the subproblems via Lagrange-multipliers. We will use the following theorem frequently without always referring to it.
\begin{theorem}[Characterization of the subproblems]\label{thm:global_solution_subproblem}
Let $f \in \mathbb R,g \in H$, $B \in \mathcal{L}(H)$ be a self-adjoint linear bounded operator and $\Delta >0$. Then the following statements are true.
\begin{itemize}
    \item[(i)] The vector $s_\star \in H$ is a global solution of the problem
\begin{equation}
\min_{s \in H} \, m(s):= f+\langle g,s\rangle+\tfrac12\langle Bs,s\rangle,
\quad \text{subject to} \quad \|s\| \leq \Delta,
\end{equation}
if and only if $\|s_\star\|\leq \Delta$  and there is a scalar $\lambda \geq 0$ such that the following conditions are satisfied:
\begin{align}
(B + \lambda I)s_\star &= -g, \notag\\
\lambda(\Delta - \|s_\star\|) &= 0, \label{eq:lagrange_system}\\
(B + \lambda I) &\text{ is positive semidefinite.} \notag
\end{align}
\item[(ii)] 
In addition if $B \succeq0$ (e.g. if $B = \nabla^2 f(x)$ for some convex $C^2$--function) and $\|s_\star\| < \Delta$, then $s_\star$ solves the unconstrained problem globally, i.e.
\begin{equation}
s_\star \in \argmin_{s \in H} m(s) = f + \langle g, s \rangle + \tfrac{1}{2} \langle B s, s\rangle  \label{eq:global_solution_subproblem} 
\end{equation}
\end{itemize}
\end{theorem}
\begin{proof}
     The proof of $(i)$ can be found in \cite[Theorem~4.3]{wright2006numerical}. For part $(ii)$ note that $\|s_\star\| < \Delta$ implies that $\lambda=0$, due to the second equation in \eqref{eq:lagrange_system}. The first equation in \eqref{eq:lagrange_system} implies that $s_\star$ solves the first-order optimality system $Bs_\star =-g$, of the problem \eqref{eq:global_solution_subproblem}. As the subproblems are convex, we have $B\succeq 0$ and consequently $Bs_\star =-g$ is equivalent to \eqref{eq:global_solution_subproblem}.
\end{proof}
Let us now state the Cauchy decrease property, which is an important basis for proving convergence of trust-region methods. For that purpose we also introduce the \emph{Cauchy direction}, $s_k^C \in H$. It leads to a point of basic descent, which is often easy to compute and which, for a given radius $\Delta_k$ and a non-stationary point (i.e. $\|g_k\|>0$), is defined as follows:
\begin{equation}
    s_k^C = -t_k^C g_k,
\qquad
t_k^C =
\arg\min_{0 \leq t \leq \Delta_k/\|g_k\|}
m_k(-t g_k). \label{eq:def_cauchy_point}
\end{equation}
In our specific setting, where $m_k$ in \eqref{eq:def_model_problem_k} is quadratic, the Cauchy direction can be calculated in closed form, since it is simply the solution of a one-dimensional quadratic minimization problem over a closed interval. Also for more general model functions, the Cauchy direction can be approximated easily; see e.g., \cite{conn2000trust}.

\begin{lemma}[Cauchy decrease]\label{lem:normal_cauchy_decrease}
Let $m_k$ be the trust-region model defined in \eqref{eq:def_model_problem_k} and denote the trust-region radius by $\Delta_k>0$.  Then the Cauchy direction $s_k^C \in H$, defined in \eqref{eq:def_cauchy_point} satisfies
\begin{equation}
m_k(0) - m_k(s_k^C)
\;\geq\;
\kappa_C \,\|g_k\| \,
\min\!\left\{
\frac{\|g_k\|}{1+\|H_k\|_{\mathrm{op}}},\, \Delta_k
\right\}  \tag{FCD}  \label{eq:standard_cauchy_decrease}
\end{equation}
for some constant $\kappa_C \in (0,1)$.
\end{lemma}
\begin{proof}
    A proof of this result can be found in \cite[Lemma~2.3.2]{cartis2022evaluation}.
\end{proof}
As a direct consequence we obtain the following lemma.
\begin{lemma}[Well-definedness and monotonicity]
\label{lem:well-definedness}
Consider either \Crefalgvariant{var:variant1}{} or \Crefalgvariant{var:variant2}{} of \Cref{alg:abstract_tr_algorithm}, and suppose that $\|g_k\|>0$. Then the following statements hold:
\begin{enumerate}
 
\item[(i)] the denominator of the acceptance ratio is strictly positive:
    \[
        p_k(s_k)
        +\frac{\vartheta}{2}
        \min\bigl\{
            \|\nabla f(x_k)\|,
            \|\nabla f(x_k+s_k)\|
        \bigr\}\|s_k\|
        >0;
    \]
    consequently, \(\rho_k^\vartheta\) is well defined;
    \item[(ii)] the objective values are monotonically nonincreasing:
    \[
        f(x_{k+1})\leq f(x_k).
    \]
    More precisely, the inequality is strict whenever iteration \(k\)
    is successful, i.e. $\rho_k^\vartheta \geq \eta_1$.
\end{enumerate}
\end{lemma}

\begin{proof}
To prove the claim, it suffices to apply \Cref{lem:normal_cauchy_decrease}. For \Crefalgvariant{var:variant1}{}, accounting for the inexactness parameter $\bar{\tau}>0$, we obtain
\begin{equation}
m_k(0)-m_k(s_k)
\geq
\frac{1}{1+\bar{\tau}}
\bigl(m_k(0)-m_k(s_k^\star)\bigr)
\geq
\frac{1}{1+\bar{\tau}}
\bigl(m_k(0)-m_k(s_k^C)\bigr)
\overset{\eqref{eq:standard_cauchy_decrease}}{>}
0,
\end{equation}
where the strict inequality follows from the assumptions $\|g_k\|>0$ and $\Delta_k>0$. Hence, $p_k(s_k)>0$ for \Crefalgvariant{var:variant1}{}. The argument for \Crefalgvariant{var:variant2}{} is analogous. Therefore, the acceptance ratio is well defined for both variants, proving~$(i)$.
For $(ii)$, if $\rho_k^\vartheta<\eta_1$, the step is rejected and the objective value remains unchanged. Otherwise,
\begin{equation*}
f(x_k)-f(x_{k+1})
\geq \eta_1\left(
p_k(s_k)
+\frac{\vartheta}{2}
\min\left\{ \|\nabla f(x_k)\|, \|\nabla f(x_k+s_k)\| \right\}\|s_k\|
\right)
\overset{(i)}{>}0.
\end{equation*}
Thus statement $(ii)$ is proven.
\end{proof}
\subsection{On the set of successful iterations}
Before starting the analysis, we shall follow the usual approach in trust-region methods and divide the set of iterations into successful and unsuccessful iterations. For this purpose, define the sets
\begin{equation*}
   \mathcal S:=\{k\in\mathbb N:\rho_k^\vartheta\ge \eta_1\}, \qquad \mathcal{U}:= \mathbb N \setminus \mathcal{S} .
\end{equation*}
We call $\mathcal{S}$ the set of \emph{successful} iterations and $\mathcal{U}$ the set of \emph{unsuccessful} iterations. Moreover, if $\rho_k^\vartheta \geq \eta_2$, we call the iteration $k$ \emph{very successful}. Let us further introduce the following sets:
\begin{equation}
    \mathcal{S}_k := \{ j \in \mathcal{S}: j \leq k\} ,\qquad \mathcal{U}_k := \{ j \notin \mathcal{S}: j \leq k\} 
\end{equation}
Let us first prove that under \Cref{ass:nonconvex_assumptions} the set of successful iterations is infinite or \Cref{alg:abstract_tr_algorithm} reaches a stationary point in a finite number of steps.
\begin{lemma}[Infinitely many successful iterations]\label{thm:infinite_successful_set}
Let \Cref{ass:nonconvex_assumptions} hold true. Then the following statements are valid.
\begin{itemize}
    \item[(i)] In the setting of \Crefalgvariant{var:variant1}{}, \Cref{alg:abstract_tr_algorithm}  either generates a first-order stationary point in finitely many iterations, or the set $\mathcal{S}$ of successful iterations is infinite. 
\item[(ii)] In the setting of \Crefalgvariant{var:variant2}{},  if in addition
\begin{equation}
0<\eta_1 <\frac{1}{1+\vartheta/2},    \label{eq:lem_succ_0} 
\end{equation}
then \Cref{alg:abstract_tr_algorithm}  either generates a first-order stationary point in finitely many iterations, or the set of successful iterations $\mathcal{S}$ is also infinite.
\end{itemize}

\end{lemma}
\begin{proof}
For the inexact trust-region method \Crefalgvariant{var:variant1}{}, the statement is classical and can be found in \cite{conn2000trust}. The proof in \cite{conn2000trust} is given under the assumption that $\|\nabla^2 f(x_k)\|_{\mathrm{op}}$ is uniformly bounded in $k$. However, this assumption is not needed here. Indeed, for a proof by contradiction, similar to the proof of $(ii)$, it suffices that $x_k = \bar{x}$ for some $\bar{x} \in H$ and all sufficiently large $k$; the required boundedness then follows immediately. \\[0.2cm]
Let us now consider \Crefalgvariant{var:variant2}{} of the algorithm: Suppose, for contradiction, that the set of successful iterations $\mathcal{S}$ is finite and that the method does not generate a first-order stationary point in finitely many iterations. Then there exists $K_0\in\mathbb N$ such that every iteration $k\geq K_0$ is unsuccessful and consequently $x_k=:\bar x$ for all $ k \geq K_0$ and  some $\bar x \in H$. Also note that $\|\nabla f( \bar x)\|>0$ for all $k\geq K_0$ since no first-order stationary point was generated in a finite number of steps. Moreover, as every iteration $k\geq K_0$ is unsuccessful, the radius update gives $\Delta_{k+1} \leq \gamma_2 \Delta_k$ with $0<\gamma_2<1$ for all $k \geq K_0$. Consequently
\begin{equation}
\Delta_k\to 0 \qquad (k \to \infty).   \label{eq:lem_succ_1} 
\end{equation}
Let us now consider the acceptance ratio \eqref{eq:rho_cat_def}. The numerator can be estimated from below via
\begin{equation}
    f(x_k) - f(x_k + s_k) = (m_k(0) - m_k(s_k)) + m_k(s_k) - f(x_k +s_k) \overset{\eqref{eq:eq_model_approx_func}}{\geq}  p_k(s_k) - C_f\Delta_k^{2+\nu},  \label{eq:lem_succ_2}
\end{equation}
where we used $p_k(s_k)$ for the predicted decrease as in \eqref{eq:def_actual_decrease_predicted_decrease}. Note that by construction $s_k$ is a minimizer of the map $ s \mapsto \langle \nabla f(\bar x),s \rangle + \frac{1}{2}\langle \nabla^2f(\bar x)s,s\rangle = - p_k(s)$ over the set $\{ s \in H:\|s\| \leq \Delta_k\}$ and consequently also a maximizer of $s\mapsto p_k(s)$ over the same set.
Thus we may compare against the feasible point $\tilde s:= -\Delta_k (\nabla f(\bar x)/ \|\nabla f(\bar x)\|)$ which yields
\begin{equation}
    p_k(s_k) 
=
\max_{\|s\|\leq \Delta_k}
\left\{
-\langle \nabla f(\bar x) ,s\rangle
-\frac12\langle \nabla^2 f(\bar x) s,s\rangle
\right\} \overset{(s= \tilde s)}{\geq } \| \nabla f(\bar x)\| \Delta_k - \frac{1}{2}\|\nabla^2 f(\bar x)\|_{\mathrm{op}} \Delta_k^2 \qquad (k \geq K_0). \label{eq:lem_succ_3}
\end{equation}
For the denominator we may estimate by Cauchy--Schwarz inequality 
\begin{equation}
\begin{aligned}
      p_k(s_k) &\leq \| \nabla f(\bar x)\| \Delta_k + \frac{1}{2}\|\nabla^2 f(\bar x)\|_{\mathrm{op}} \Delta_k^2 \\
      \min\{\| \nabla f(\bar x)\|,\| \nabla f(\bar x + s_k)\| \} \|s_k\| &\leq \| \nabla f(\bar x)\| \Delta_k\label{eq:lem_succ_4}
\end{aligned}
\end{equation}
for all $k \geq K_0$. As $\Delta_k \to 0$, we may enlarge $K_0$ such that additionally
\begin{equation*}
    \| \nabla f(\bar x)\| \Delta_k - \frac{1}{2}\|\nabla^2 f(\bar x)\|_{\mathrm{op}}\Delta_k^2- C_f\Delta_k^{2+\nu} >0 \qquad \text{for all $k \geq K_0$.}
\end{equation*}
Now combining the previous estimates, we obtain for the acceptance ratio
\begin{align*}
     \rho_k^\vartheta = \frac{f(x_k) - f(x_{k}+s_k)}{p_k(s_k)+ \frac{\vartheta}{2}\min\{ \|\nabla f(x_k)\|,\|\nabla f(x_{k,+})\|\} \|s_k\|} \overset{\eqref{eq:lem_succ_2},\eqref{eq:lem_succ_3},\eqref{eq:lem_succ_4}}{\geq } \frac{\| \nabla f(\bar x)\| \Delta_k - \frac{1}{2}\|\nabla^2 f(\bar x)\|_{\mathrm{op}}\Delta_k^2- C_f\Delta_k^{2+\nu}}{\| \nabla f(\bar x)\| \Delta_k + \frac{1}{2}\|\nabla^2 f(\bar x)\|_{\mathrm{op}}\Delta_k^2 + \frac{\vartheta}{2} \| \nabla f(\bar x)\| \Delta_k}.
\end{align*}
Recalling \eqref{eq:lem_succ_1} and letting $k \to \infty$, we infer after division by $\Delta_k\|\nabla f(\bar x)\|>0$ in both the denominator and numerator that
\begin{equation*}
    \rho_k^\vartheta 
    \geq
    \frac{ 1 - \frac{(\frac{1}{2}\|\nabla^2 f(\bar x)\|_{\mathrm{op}} \Delta_k+ C_f\Delta_k^{1+\nu})}{\|\nabla f(\bar x)\|}}{1 + \frac{\frac{1}{2}\|\nabla^2 f(\bar x)\|_{\mathrm{op}} \Delta_k}{\| \nabla f(\bar x)\|}  + \frac{\vartheta}{2}} \to \frac{1}{1 + \frac{\vartheta}{2}} \qquad (k \to \infty) .
\end{equation*}
By assumption \eqref{eq:lem_succ_0}, we deduce that eventually $\rho_k^\vartheta \geq \eta_1$ for sufficiently large  $k \geq K_0$, which contradicts the assumption that there are no successful iterations after $K_0$.
\end{proof}

\paragraph{Standing convention.}\label{par:standing_convention}
By \Cref{thm:infinite_successful_set}, \Cref{alg:abstract_tr_algorithm} either reaches a first-order stationary point in finitely many iterations, or the set $\mathcal{S}$ of successful iterations is infinite. The finite-termination case is harmless for the complexity estimates in this work: in the nonconvex regime the desired first-order stationarity has already been reached, while under convexity every stationary point is a global minimizer. We therefore work, from now on, in the non-terminating case. More precisely, we assume 
\[
    \|\nabla f(x_k)\|>0,
    \qquad
    f(x_k)-f_\star>0
    \qquad \text{for all } k\in\mathbb N .
\]
In particular all quantities involving \(F_0^{-\alpha}\) are understood under this convention and are well defined.

\subsection{Global iteration complexity guarantees}
We briefly recall what we mean by complexity guarantees. For two nonnegative sequences $(a_k)_{k \in \mathbb N}$ and $(b_k)_{k \in \mathbb N}$, throughout the paper, the notation $a_k=\mathcal O(b_k)$ means that there exist
constants $C>0$ and $k_0\in\mathbb N$, independent of $k$, such that
\[
a_k\le C b_k \qquad \text{for all } k\geq k_0 .
\]
Similarly, for complexity bounds depending on an accuracy parameter
$\varepsilon>0$, the notation $K(\varepsilon)= \mathcal O(\phi(\varepsilon))$ as
$\varepsilon\downarrow0$ means that there exist constants $C>0$ and
$\varepsilon_0>0$, independent of $\varepsilon$, such that
\[
K(\varepsilon)\le C\phi(\varepsilon)
\qquad \text{for all } \varepsilon\in(0,\varepsilon_0].
\]
Unless stated otherwise, the hidden constants may depend on the problem data
and on the algorithmic parameters, but not on the target accuracy $\varepsilon$ or on the iteration counter. \\

The following theorem enables us to use bounds on the complexity on the subset of successful iterations to derive global iteration complexities with respect to all iterations. It is standard in this regard for trust-region analysis, see \cite{cartis2022evaluation}. 
\begin{lemma}[General complexity bound]\label{thm:relation_successful_unsuccessful}
    Let $k \geq 1$ and consider the $k^{\mathrm{th}}$ iteration of \Cref{alg:abstract_tr_algorithm}. Also let $\Delta_\varepsilon>0$ be a real number such that $\Delta_\varepsilon\leq \Delta_{k}$. Then the following inequality holds 
    \begin{align}
    k
    \leq
    \left(
        1+\frac{\log(\gamma_3)}{|\log(\gamma_2)|}
    \right)
    |\mathcal{S}_{k-1}|
    +
    \frac{1}{|\log(\gamma_2)|}
     \log_+\left( 
        \frac{\Delta_0 }{\Delta_\varepsilon}
    \right)\notag.
\end{align}
Here we used the notation $\log_+(x):= \max(0,\log(x))$ for the positive part of the $\log$--function.
\end{lemma}
\begin{remark}
Here and during the whole manuscript, if not stated otherwise, we use the convention $\log(x) = \log_e(x)$.    
\end{remark}

\begin{proof}
  A similar proof is given in \cite[Lemma~2.3.1]{cartis2022evaluation}. Since the indexing is slightly different here, we repeat the argument for the reader's convenience. Let $k  \geq 1$ be given. By the radius update mechanism, for every $i \in \mathbb N$
  \begin{equation*}
      \begin{dcases}
          \Delta_{i+1} \leq \gamma_2\Delta_i &\text{if $i \notin \mathcal{S}$}\\
          \Delta_{i+1} \leq  \gamma_3\Delta_i & \text{if $i \in \mathcal{S}$}.
      \end{dcases} 
  \end{equation*}  
  By repeatedly applying these bounds from $i=0$ to $i=k-1$, and taking into account $\Delta_\varepsilon \leq \Delta_k$, we obtain
\begin{equation*}
\Delta_\varepsilon \leq \Delta_k
\leq
\Delta_0\gamma_3^{|\mathcal{S}_{k-1}|}\gamma_2^{|\mathcal{U}_{k-1}|}. 
\end{equation*} 
Dividing by $\Delta_0>0$, taking $\log$ on both sides and taking into account that $\log(\gamma_2)<0$, we infer 
\begin{equation}
| \mathcal{U}_{k-1}|
\leq \left(
\frac{\log(\gamma_3)}{|\log(\gamma_2)|} \right) | \mathcal{S}_{k-1}|
+
\frac{1}{|\log(\gamma_2)|}
\log\left(\frac{\Delta_0}{\Delta_{\varepsilon}}\right).\label{eq:successful_bound_eq1}
\end{equation}
Since every index $i\in\{0,\ldots,k-1\}$ is either successful or unsuccessful, we have
\[
|\mathcal{S}_{k-1}| + |\mathcal{U}_{k-1}| = k .
\]
Therefore combining this with \eqref{eq:successful_bound_eq1} yields,
\[
k
\leq
\left(
1+\frac{\log(\gamma_3)}{|\log(\gamma_2)|}
\right)|\mathcal S_{k-1}|
+
\frac{1}{|\log\gamma_2|}
\log_+\left(\frac{\Delta_0}{\Delta_{\varepsilon}}\right),
\]
which proves the claim.
\end{proof}

\section{A universal complexity mechanism for the convex regime}\label{sec:abstract_convexity_thm}
We now turn to one of the main contributions of this article: the universal
complexity analysis in the convex regime. The aim of this section is to develop
an abstract trust-region complexity theorem that applies to both
\Crefalgvariant{var:variant1}{} and \Crefalgvariant{var:variant2}{} of
\Cref{alg:abstract_tr_algorithm}. The theorem isolates the two ingredients
needed for the analysis: a function-gap model-decrease property and the fact
that sufficiently small trust-region radii lead to very successful iterations.
Once these two properties are verified, the abstract framework yields the
universal convex complexity bound
\[
    \mathcal{O}\bigl(\varepsilon^{-1/(1+\nu)}\bigr)
\]
without requiring knowledge of the H\"older exponent \(\nu\).
\subsection{A model decrease property}
The following result is the central estimate for this analysis. It is reminiscent of the classical Cauchy-decrease property \eqref{eq:standard_cauchy_decrease}; however, instead of a decrease estimate in terms of the gradient norm, it provides a model-decrease estimate in terms of the scaled function gap \((f(x_k)-f_\star)^\alpha\). This estimate will eventually enable us to prove the desired complexity bounds. The proof is inspired by \cite{cartis2012evaluation}.

First, note that by the standing convention in \Cref{par:standing_convention} the expression \(\Delta_0/F_0^\alpha\) is well-defined.

\begin{lemma}[Function-gap model-decrease property]\label{lem:convex_model_decrease_alpha}
    Let \Cref{ass:nonconvex_assumptions} and \Cref{ass:convex_assumptions} hold true and consider the $k^{\mathrm{th}}$ fixed iteration of \Cref{alg:abstract_tr_algorithm} with current iterate $x_k \in H$. Let $\alpha>0$ as in \eqref{eq:def_C_f_C_g_alpha} and define the constants
    \begin{equation}
        \beta:=\min\left\{ \frac{D}{F_0^{\alpha}},
\left(\frac{1-\eta_2}{2C_f D}\right)^{\alpha}
\right\}.\label{eq:defn_beta_C}
    \end{equation}
For a trial radius $\Delta>0$ let $s_k(\Delta)$ denote a global minimizer of the subproblem
\begin{equation}
\min_{\|s\|\leq \Delta} m_k(s)=  f(x_k) + \langle\nabla  f(x_k),s\rangle + \frac12\langle \nabla^2 f(x_k)s,s\rangle.  \label{eq:fcd_thm_subproblem}  
\end{equation}
Then the following \emph{model-decrease} inequality is valid 
\begin{equation}
m_k(0)-m_k(s_k(\Delta)) \geq \frac{(1+\eta_2)F_k}{2D}\min\left \{\beta F_k^{\alpha},\,\Delta \right \}.  \tag{$\mathrm{MD}_\alpha$} \label{eq:md_alpha}
\end{equation}
\end{lemma}
\begin{proof}
Let $x_k$ be the current iterate. Choose $x_k^\star = \mathrm{proj}_{X_\star}(x_k) \in X_\star$. Note that the projection is well defined as by convexity and continuity of $f$, the solution set $X_\star$ is closed and convex. Now consider the residual
\begin{equation*}
d_k^\star:=x_k^\star-x_k.
\end{equation*}
As clearly $x_k,x_k^\star \in \mathcal{L}_0 = \{x: f(x) \leq f(x_0)\}$, by assumption \eqref{eq:assumption_sublevel}, $\|d_k^\star\| = \mathrm{dist}(x_k,X_\star)\leq D$. We define
\begin{equation*}
t:=\min\left\{\frac{\beta F_k^\alpha}{D},\ \frac{\Delta}{D}\right\},
\qquad
\hat s:= t\,d_k^\star.    
\end{equation*}
Since $F_k\leq F_0$ and $\beta\leq D/F_0^\alpha$ we deduce
\begin{equation*}
0\leq t\leq \frac{\beta F_k^\alpha}{D}\leq \frac{\beta F_0^\alpha}{D}\leq 1.    
\end{equation*}
Also, since $t\leq \Delta/D$ and $\|d_k^\star\|\leq D$, we have
\begin{equation*}
\|\hat s\|=t\|d_k^\star\|\le t D \leq \Delta.
\end{equation*}
Hence $\hat s$ is a feasible point for the trust-region subproblem \eqref{eq:fcd_thm_subproblem} and consequently
\begin{equation}
    m_k(s_k(\Delta))\leq m_k(\hat s). \label{eq:fcd_thm_min_s_hat}
\end{equation}
Now note that $x_k+\hat s=(1-t)x_k+t x_k^\star$, so by convexity of $f$ and $t \in [0,1]$,
\begin{equation}
f(x_k+\hat s)\leq (1-t)f(x_k)+ t f(x_k^\star) = f(x_k) - t F_k.
\label{eq:fcd_thm_convex_t}
\end{equation}
Using \Cref{lem:model_approximation},
\begin{equation*}
m_k(\hat s)\le f(x_k + \hat s)+C_f \|\hat s\|^{2+\nu},   
\end{equation*}
which in combination with \eqref{eq:fcd_thm_convex_t} gives
\begin{equation}
f(x_k) - m_k(\hat s)\geq t F_k-C_f \| \hat s\|^{2+\nu}. \label{eq:fcd_thm_descent_1}
\end{equation}
Since $\|\hat s\|\leq tD$ and $t\leq \beta F_k^\alpha/D$ by construction,
\begin{equation*}
C_f \|\hat s\|^{2+\nu}\leq C_f D^{2+\nu} t^{2+\nu} =
C_f D^{1+\nu} t^{1+\nu}Dt
\leq
C_f D \beta^{1+\nu} F_k\, t,
\end{equation*}
where we used $\alpha(1+\nu)=1$ in the last step. By the definition of $\beta$,
\begin{equation*}
C_f D \beta^{1+\nu}\leq \frac{1-\eta_2}{2},
\end{equation*}
which eventually leads to
\begin{equation}
C_f \|\hat s\|^{2+\nu}\leq \left(\frac{1-\eta_2}{2} \right) F_k t. 
\label{eq:fcd_thm_descent_2}
\end{equation}
Combining the above estimates with \eqref{eq:fcd_thm_descent_1}, we obtain the final result:
\begin{equation*}
 m_k(0) - m_k(s_k(\Delta)) \overset{\eqref{eq:fcd_thm_min_s_hat}}{\geq}
 m_k(0) - m_k(\hat s) \overset{\eqref{eq:fcd_thm_descent_1},\eqref{eq:fcd_thm_descent_2}}{\geq}\left(
 \frac{1+\eta_2}{2}\right) F_k t = \left(\frac{1+\eta_2}{2} \right)F_k \min\left\{\frac{\beta F_k^\alpha}{D},\ \frac{\Delta}{D}\right\}.
\end{equation*}
In the last equation we simply used the definition of $t$. 
\end{proof}
\begin{remark}\label{rem:testability_of_fcd}
Note that, while useful in theory, in practical applications the condition \eqref{eq:md_alpha} is not directly checkable unless $f_\star$ is known a priori.

In order to obtain a checkable descent condition in the setting of \Cref{lem:convex_model_decrease_alpha}, a natural first attempt would be to use the inequality
\begin{equation*}
F_k \leq D \|\nabla f(x_k)\|,
\end{equation*}
which follows from convexity and Assumption \eqref{eq:assumption_sublevel}. One would then try to formulate a checkable decrease condition, in the spirit of the classical Cauchy decrease as in \Cref{lem:normal_cauchy_decrease}, involving the gradient:
\begin{equation*}
m_k(0)-m_k(s_k(\Delta)) \geq \kappa_1 \|\nabla f(x_k)\| \min\left \{\kappa_2 \|\nabla f(x_k)\|^{\alpha},\Delta \right \}
\end{equation*}
for some constants $\kappa_1,\kappa_2 >0$. However, it is easy to check that such a condition, for arbitrary $\Delta>0$, can hold only when $\alpha=1$. The simple example $f(x) = \frac{1}{2}x^2$ already reveals this obstruction.
\end{remark}

\subsection{Main theorem}
We will now present our main theorem which will enable us to prove universal iteration complexity bounds for \Cref{alg:abstract_tr_algorithm}.

\begin{theorem}[Abstract convex complexity theorem]\label{thm:main_abstract_convex_result}
Let \Cref{ass:convex_assumptions} hold true. Again denote by $s_k^\star = s_k(\Delta_k) \in H$ the global solution of the subproblem \eqref{eq:fcd_thm_subproblem}. Moreover let $\alpha,\beta>0$ be the constants in \eqref{eq:def_C_f_C_g_alpha}, \eqref{eq:defn_beta_C}. Also define the hitting index
\begin{equation*}
    k_{\mathrm{cvx}}(\varepsilon):= \min \{k \in\mathbb N: f(x_{k})- f_\star \leq \varepsilon\}.
\end{equation*}
Assume that the following conditions hold:
\begin{itemize}
    \item[(i)] There exists a $\kappa_{\mathrm{cvx}}\in(0,1]$ such that
    \begin{equation}
        m_k(0) - m_k(s_k) \geq \kappa_{\mathrm{cvx}}(m_k(0) - m_k(s_k^\star)) \qquad \text{for all $k \in \mathbb N$}.\label{eq:cvx_abstract_thm_fractional_model_descent}
    \end{equation}
    \item[(ii)] There exists a constant $\beta_{\mathrm{cvx}} \in (0,\beta])$, where 
$\beta>0$ is defined in \eqref{eq:defn_beta_C}, such that
\begin{equation}
    \Delta_k \leq \beta_{\mathrm{cvx}} F_k^\alpha 
    \quad \Longrightarrow \quad 
    \rho_k^\vartheta \geq \eta_2 \qquad \text{for all $k \in \mathbb N$}. \label{eq:cvx_abstract_thm_lower_bound_radii} 
\end{equation}
Equivalently, sufficiently small radii yield very successful iterations.
\end{itemize}
In addition define the constants
\begin{equation}
\mu_{\mathrm{cvx}}:= \min\left\{
        \gamma_1\beta_{\mathrm{cvx}},
        \frac{\Delta_0}{F_0^\alpha}
    \right\}, \qquad 
    C_{\mathrm{cvx}}
    :=
    \frac{\eta_1(1+\eta_2) \kappa_{\mathrm{cvx}}\mu_{\mathrm{cvx}}}{2D}. \label{eq:cvx_abstract_thm_def_mu_C}
\end{equation}
Then the following iteration complexity bound holds:
\begin{equation*}
    k_{\mathrm{cvx}}(\varepsilon) 
    \leq
    \left(
        1+\frac{\log(\gamma_3)}{|\log(\gamma_2)|}
    \right)
    \left(
        \frac{1}{\alpha C_{\mathrm{cvx}}\varepsilon^{1/(1+ \nu)}}
    \right)
    +
    \frac{1}{|\log(\gamma_2)|} 
    \log_+\left(
        \frac{\Delta_0 }{\mu_{\mathrm{cvx}}\,\varepsilon^{1/(1+ \nu)}} 
    \right) + 1 =  \mathcal{O}\left( \varepsilon^{-\frac{1}{1+\nu}} \right) ,
\end{equation*}
as $\varepsilon \downarrow 0$.
\end{theorem}
We divide the proof into several lemmas. The key step in establishing a convergence rate on the subset of successful iterations is to lower bound the accepted radius $\Delta_k$ so that the first argument in $\min(\cdot,\cdot)$ becomes active in \eqref{eq:md_alpha}. This is achieved by the following lemma.

\begin{lemma}[Lower bound on the trust-region radii in the convex regime]
\label{lem:cvx_abstract_thm_lower_bound_radii}
Let the assumptions of \Cref{thm:main_abstract_convex_result} hold and assume that there is a $\beta_{\mathrm{cvx}}\leq \beta$, with $\alpha,\beta>0$ defined in \eqref{eq:def_C_f_C_g_alpha} and  \eqref{eq:defn_beta_C}. Then we have the following lower bound:
\begin{equation}
    \Delta_k
    \geq
     \min\left\{
        \gamma_1\beta_{\mathrm{cvx}},
        \frac{\Delta_0}{F_0^\alpha}
    \right\} F_k^{\alpha} \qquad \text{for all $k \in \mathbb N$.}
    \label{eq:accepted_radius_lower_bound_abstract}
\end{equation}
\end{lemma}

\begin{proof}
    We prove \eqref{eq:accepted_radius_lower_bound_abstract} by induction over $k \in \mathbb N$. If $k=0$ then we have automatically
\begin{equation*}
    \Delta_0 = \frac{\Delta_0}{F_0^\alpha} F_0^\alpha \geq \min\left\{
        \gamma_1\beta_{\mathrm{cvx}},
        \frac{\Delta_0}{F_0^\alpha}
    \right\} F_0^{\alpha},
\end{equation*}
which is \eqref{eq:accepted_radius_lower_bound_abstract} for $k=0$. We now assume that \eqref{eq:accepted_radius_lower_bound_abstract} holds true for some $k \in \mathbb N$. In order to prove \eqref{eq:accepted_radius_lower_bound_abstract}  for the next iteration $(k + 1)$, we distinguish two cases. \\[0.2cm]
\emph{Case 1:} $\rho_k^\vartheta < \eta_2$. This means that iteration \(k\) is not very successful. Then by assumption \eqref{eq:cvx_abstract_thm_lower_bound_radii}, we conclude $\Delta_k > \beta_{\mathrm{cvx}} F_k^\alpha$. Therefore the radius-update mechanism gives
\begin{equation*}
\Delta_{k+1} \geq \gamma_1 \Delta_k >  \gamma_1 \beta_{\mathrm{cvx}} F_k^\alpha 
\geq
\min\left\{
        \gamma_1\beta_{\mathrm{cvx}},
        \frac{\Delta_0}{F_0^\alpha}
    \right\} F_k^{\alpha} 
    \geq
    \min\left\{
        \gamma_1\beta_{\mathrm{cvx}},
        \frac{\Delta_0}{F_0^\alpha}
    \right\} F_{k+1}^{\alpha},
\end{equation*}
where we used that $F_{k+1} \leq F_k$ in the last inequality. This proves \eqref{eq:accepted_radius_lower_bound_abstract} for $(k+1) $. \\[0.2cm]
\emph{Case 2:}  $\rho_k^\vartheta \geq \eta_2$, so the iteration $k$ is very successful. Hence the radius update mechanism again gives in this case
\begin{equation*}
\Delta_{k+1} \geq \Delta_k  \overset{\rm (IH)}\geq
\min\left\{
        \gamma_1\beta_{\mathrm{cvx}},
        \frac{\Delta_0}{F_0^\alpha}
    \right\} F_k^{\alpha} 
    \geq
    \min\left\{
        \gamma_1\beta_{\mathrm{cvx}},
        \frac{\Delta_0}{F_0^\alpha}
    \right\} F_{k+1}^{\alpha},
\end{equation*}
where (IH) is the induction hypothesis and in the last inequality we again used $F_{k+1} \leq F_k$. Combining both cases, we conclude that the statement of the theorem is true for all $k\in \mathbb N$.
\end{proof}
\begin{remark}The proof above does not use the fact that $\beta_{\mathrm{cvx}} \leq \beta$. All steps are valid without this restriction.    
\end{remark}
The next theorem establishes the desired convergence rate on the set of successful iterations.
The main argument in the proof of the following theorem is standard in the analysis of convex first-order methods; see, e.g., \cite{nesterov2018lectures}.

\begin{lemma}[Rate on successful steps]
\label{lem:successful_rate_abstract}
Let the assumptions of \Cref{thm:main_abstract_convex_result} hold. Recall the constants
\begin{equation*}
\mu_{\mathrm{cvx}}= \min\left\{
        \gamma_1\beta_{\mathrm{cvx}},
        \frac{\Delta_0}{F_0^\alpha}
    \right\}, \qquad 
    C_{\mathrm{cvx}}
    =
    \frac{\eta_1(1+\eta_2) \kappa_{\mathrm{cvx}}\mu_{\mathrm{cvx}}}{2D}.
\end{equation*}
Then the following inequality holds for all $k \geq 1$:
\begin{equation}
    f(x_{k}) - f_\star \leq \frac{1}{(F_0^{-\alpha}+\alpha C_{\mathrm{cvx}} |\mathcal{S}_{k-1}|)^{1/\alpha}}. \label{eq:convex_abstract_rate_successful}
\end{equation}
\end{lemma}

\begin{proof}
By the fractional model decrease \eqref{eq:cvx_abstract_thm_fractional_model_descent} we obtain for all $n \in \mathbb N$,
\begin{equation*}
    m_{n}(0)-m_{n}(s_{n})
    \overset{\eqref{eq:cvx_abstract_thm_fractional_model_descent}}{\geq} 
    \kappa_{\mathrm{cvx}}\left(
    m_{n}(0)-m_{n}(s_{n}^\star) \right)
    \overset{\eqref{eq:md_alpha}}{\geq }
   \frac{\kappa_{\mathrm{cvx}}(1+\eta_2) F_n}{2D}
    \min\{\beta_{\mathrm{cvx}} F_n^\alpha , \Delta_n\},
\end{equation*}
where we also used $\beta \geq \beta_{\mathrm{cvx}}$ in the last inequality and that $s_{n}^\star$ solves the subproblem \eqref{eq:fcd_thm_subproblem}. By \Cref{lem:cvx_abstract_thm_lower_bound_radii}, we also have 
\begin{equation}
    \Delta_n
    \geq
    \min\left\{
        \gamma_1\beta_{\mathrm{cvx}},
        \frac{\Delta_0}{F_0^\alpha}
    \right\}F_n^\alpha = \mu_{\mathrm{cvx}} F_n^\alpha \qquad \text{for all $n \in \mathbb N$. } \label{eq:successful_abstract_1}
\end{equation}
 Since $\gamma_1\beta_{\mathrm{cvx}}\leq \beta_{\mathrm{cvx}}$ we infer using  monotonicity of $\min(\cdot,\Delta_n)$, that
\begin{align}
    m_{n}(0)-m_{n}(s_{n}) 
    &\geq \left(
    \frac{\kappa_{\mathrm{cvx}}(1+\eta_2) F_n}{2D}\right)
    \min \left \{\mu_{\mathrm{cvx}}F_n^\alpha , \Delta_n \right \} 
    \overset{\eqref{eq:successful_abstract_1}}{=}
    \left(
     \frac{\kappa_{\mathrm{cvx}}(1+\eta_2)\mu_{\mathrm{cvx}}}{2D} \right)F_n^{1+\alpha} \quad \text{for $n \in \mathbb N.$}
\label{eq:successful_abstract_11}
\end{align}
Now for successful indices $n\in  \mathcal{S}$ we have $f(x_{n})-f(x_{n}+s_{n}) \geq
    \eta_1 (m_{n}(0)-m_{n}(s_{n}))$ and therefore
\begin{align}
    F_n-F_{n+1} =
    f(x_{n})-f(x_{n}+s_{n}) \geq
    \eta_1 (m_{n}(0)-m_{n}(s_{n}))
\overset{\eqref{eq:successful_abstract_11}}{\geq}
    \left(
    \frac{\eta_1\kappa_{\mathrm{cvx}}(1+\eta_2)\mu_{\mathrm{cvx}}}{ 2D} \right)   F_n^{1+\alpha} \qquad \text{for all $n \in \mathcal{S}$}.\label{eq:inexact_successful_decrease}
\end{align}
For $x,y\geq 0$ and $\alpha \in (0,1]$ we have $x^\alpha - y^\alpha \geq \alpha x^{\alpha-1} (x-y)$ by concavity of $x \mapsto x^\alpha$. Also recall that without loss of generality, we may assume that $F_{n}>0$ for all $n \in \mathbb N$. These two facts imply
\begin{align}
\frac{1}{F_{{n+1}}^\alpha} - \frac{1}{F_{{n}}^\alpha} 
=
\frac{F_{n}^\alpha-F_{{n+1}}^\alpha}{F_{{n}}^\alpha F_{n+1}^\alpha} 
\geq
\frac{\alpha (F_{n} - F_{n+1}
)}{F_{n}F_{n+1}^{\alpha}} 
\overset{\eqref{eq:inexact_successful_decrease}}{\geq} 
\frac{\alpha C_{\mathrm{cvx}} F_{{n}}^{1+\alpha}}{F_n F_{{n+1}}^\alpha} \geq \alpha C_{\mathrm{cvx}} \qquad \text{for all $n \in \mathcal{S}$}, \label{eq:successful_compl_1_inexact}
\end{align}
where we used $F_{n+1} \leq F_n$ in the last inequality. Summing \eqref{eq:successful_compl_1_inexact} from $n=0,\ldots,k - 1$ yields
\begin{equation}
\frac{1}{F_{{k}}^\alpha} - \frac{1}{F_{{0}}^\alpha} = 
    \sum_{n = 0}^{k-1}\left(\frac{1}{F_{{n+1}}^\alpha} - \frac{1}{F_{{n}}^\alpha}  \right) \overset{(a)}{=} \sum_{n \in  \mathcal{S}_{k-1}}\left(\frac{1}{F_{{n+1}}^\alpha} - \frac{1}{F_{{n}}^\alpha}  \right) \overset{\eqref{eq:successful_compl_1_inexact}}{\geq} \alpha C_{\mathrm{cvx}} |\mathcal{S}_{k-1}|,  \qquad (k\geq 1),\label{eq:successful_compl_2_inexact}
\end{equation}
where we used in $(a)$ that $x_{n+1} = x_n$ for $n \notin \mathcal{S}$. Rearranging \eqref{eq:successful_compl_2_inexact} gives eventually
\begin{equation*}
    f(x_{k}) - f_\star \leq
    \frac{1}{(F_0^{-\alpha}+\alpha C_{\mathrm{cvx}} |\mathcal{S}_{k-1}|)^{1/\alpha}}, \qquad (k\geq 1),
\end{equation*}
which proves the theorem.
\end{proof}

We can now give a simple proof of the main theorem
\begin{proof}[Proof of \Cref{thm:main_abstract_convex_result}]
First note, that due to \Cref{lem:successful_rate_abstract} and the standing convention in \Cref{par:standing_convention}, we have $F_k \to 0$ as $k\to \infty$. Thus $k_{\mathrm{cvx}}(\varepsilon) < + \infty$ for every $\varepsilon>0$. 
Second, for $k_{\mathrm{cvx}}(\varepsilon) \leq 1$ the statement of the theorem is true, as the right hand side of the complexity inequality is strictly larger than $1$ for any $\varepsilon>0$. Hence assume throughout the proof that $k_{\mathrm{cvx}}(\varepsilon)\geq 2$. We also define
\begin{equation*}
    K(\varepsilon):= k_{\mathrm{cvx}}(\varepsilon)-1, \qquad 
    \mathcal{S}(\varepsilon):= \mathcal{S}_{K_{}(\varepsilon)-1}.
\end{equation*}
Both quantities are well-defined as $K(\varepsilon)\geq 1$. We then may apply \Cref{lem:successful_rate_abstract} with $k = K_{}(\varepsilon)\geq 1$ and obtain by the minimality of $k_{\mathrm{cvx}}(\varepsilon)$ that 
\begin{equation*}
    \varepsilon \leq f(x_{K_{}(\varepsilon)}) - f_\star \overset{\eqref{eq:convex_abstract_rate_successful}}{\leq} \frac{1}{(F_0^{-\alpha}+\alpha C_{\mathrm{cvx}}|\mathcal{S}(\varepsilon)|)^{1/\alpha}}.
\end{equation*}
 Rearranging and estimating from above yields
\begin{equation*}
    |\mathcal{S}(\varepsilon)| \leq \frac{1}{\alpha C_{\mathrm{cvx}}\varepsilon^{\alpha}}.
\end{equation*}
Now, additionally from \Cref{lem:cvx_abstract_thm_lower_bound_radii} we obtain  
\begin{equation*}
     \mu_{\mathrm{cvx}}\varepsilon^{\alpha} \leq \mu_{\mathrm{cvx}} \left( f(x_{K_{}(\varepsilon)}) - f_\star \right)^\alpha \overset{\eqref{eq:accepted_radius_lower_bound_abstract}}{\leq} \Delta_{K_{}(\varepsilon)}.
\end{equation*}
Combining the previous two estimates with \Cref{thm:relation_successful_unsuccessful} in which we chose  $ k =  K_{}(\varepsilon)$ and 
$\Delta_{\varepsilon} = \mu_{\rm cvx} \varepsilon^\alpha$  we infer
\begin{equation*}
    K_{}(\varepsilon)
    \leq
    \left(
        1+\frac{\log(\gamma_3)}{|\log(\gamma_2)|}
    \right)
    \left(
        \frac{1}{\alpha C_{\mathrm{cvx}}\varepsilon^\alpha}
    \right)
    +
    \frac{1}{|\log(\gamma_2)|}
    \log_+ \left(
        \frac{\Delta_0 }{\mu_{\mathrm{cvx}}\,\varepsilon^\alpha}
    \right).
\end{equation*}
Now the statement of the theorem follows from the relation $k_{\mathrm{cvx}}(\varepsilon)= K(\varepsilon) + 1$.    
\end{proof}
\begin{remark}
    The proof of the statement that
    $k_{\mathrm{cvx}}(\varepsilon)<+\infty$ in
    \Cref{thm:main_abstract_convex_result} makes use of the standing convention in \Cref{par:standing_convention}, which ensures that
    $|\mathcal{S}|=+\infty$ if \Cref{alg:abstract_tr_algorithm} does not converge in finitely many iterations to a stationary point. The standing convention is
    justified by \Cref{thm:infinite_successful_set} for \Crefalgvariant{var:variant1}{}, and also 
    for \Crefalgvariant{var:variant2}{} if, in addition,
    $0 <\eta_1 < 1/(1+\vartheta/2)$.

    We emphasize that the statements of this section remain valid under the assumptions of  \Cref{thm:main_abstract_convex_result} without explicitly assuming the standing convention. Indeed, if \Cref{alg:abstract_tr_algorithm} does not converge in finitely many iterations, condition
    \eqref{eq:cvx_abstract_thm_lower_bound_radii} directly implies that $|\mathcal{S}|=+\infty$. To see this, suppose that $|\mathcal{S}|<+\infty$. Then there
    exist $K_0\in\mathbb N$, $\bar x\in H$, and
    $\bar F>0$ such that
    \[
        x_k=\bar x,
        \qquad
        F_k=\bar F
        \qquad
        \text{for all }k\geq K_0.
    \]
    Indeed, after the last successful iteration all subsequent iterations are unsuccessful and the iterate remains unchanged. The radius-update mechanism then yields $\Delta_k\to 0$. Hence, for all sufficiently
    large \(k\geq K_0\),
    \[
        \Delta_k
        \leq
        \beta_{\mathrm{cvx}}\bar F^\alpha
        =
        \beta_{\mathrm{cvx}}F_k^\alpha.
    \]
    Condition \eqref{eq:cvx_abstract_thm_lower_bound_radii} therefore
    implies $\rho_k^\vartheta\geq\eta_2\geq\eta_1$,
     so iteration \(k\) is successful. This is a contradiction.
\end{remark}
\section{Analysis of the inexact basic trust-region method}\label{sec:tr_complexity} 
In this section, we study complexity bounds for \Crefalgvariant{var:variant1}{} of \Cref{alg:abstract_tr_algorithm}, an inexact version of the basic trust-region method studied in \cite{conn2000trust,cartis2022evaluation}.

More precisely, we first recall the classical first-order complexity result in \Cref{thm:tr-nonconvex-complexity}, which is well known in the literature. We then prove a universal complexity bound in \Cref{thm:tr_convex_complexity}, which holds without prior knowledge of the parameter $\nu \in [0,1]$. Finally, in \Cref{thm:tr_local_fast_convergence}, we show that once the iterates of \Crefalgvariant{var:variant1}{} enter the local convergence regime, they converge at the same universal rate as the classical Newton method.

The uniform boundedness of the Hessian in \Cref{ass:bounded_hessian} is only needed for the nonconvex complexity result, \Cref{thm:tr-nonconvex-complexity}. We expect that this assumption can be relaxed.

We also note that the nonconvex and locally strongly convex cases have already been treated in the literature.

\subsection{First-order complexity in the nonconvex regime}\label{sec:tr_first_order_complexity}
For completeness we will quote here the first-order complexity guarantees for \Crefalgvariant{var:variant1}{} in the nonconvex regime, precisely under \Cref{ass:nonconvex_assumptions} and \Cref{ass:bounded_hessian}.
\begin{theorem}[First-order iteration complexity in the nonconvex regime]\label{thm:tr-nonconvex-complexity}
Suppose that \Cref{ass:nonconvex_assumptions} and \Cref{ass:bounded_hessian} hold true. And define the quantity
\begin{equation*}
    k_{\mathrm{TR},\mathrm{ncvx}}(\varepsilon):= \min \{ k \in \mathbb N: \| \nabla f(x_k)\| \leq \varepsilon \}.
\end{equation*}
Then there are constants $\kappa_1,\kappa_2,\kappa_3>0$ such that for every \(\varepsilon>0\) we have
\[
k_{\mathrm{TR},\mathrm{ncvx}}(\varepsilon)
\leq
\kappa_1\left ( \frac{f(x_0)-f_{\star}}{\varepsilon^2} \right )
+
\kappa_2|\log\varepsilon|
+
\kappa_3
=
\mathcal{O}(\varepsilon^{-2})
 \qquad (\varepsilon\downarrow0) \]
\end{theorem}
\begin{proof}
    The complexity bound is a simplified form of a result for first-order trust-region methods in \cite[Theorem~2.3.7]{cartis2022evaluation}. Note that for the proof in \cite{cartis2022evaluation} Lipschitz continuity of $\nabla f$ is required. However, this is given here under the present \Cref{ass:bounded_hessian}, see also  \eqref{eq:Lip_cont_grad_f}.
\end{proof}

\subsection{Universal complexity in the convex regime}\label{sec:tr_convex_complexity_basic}
We now turn to one of the main contributions of this article: the complexity analysis of the classical trust-region method under convexity. To apply \Cref{thm:main_abstract_convex_result}, in particular condition (ii), we need to show that, eventually, sufficiently small radii are always very successful. This is established by the following theorem.
\begin{lemma}[Small radii lead to very successful steps, \Crefalgvariant{var:variant1}{} ]\label{lem:small_radi_accepted_ver1}
Let \Cref{ass:nonconvex_assumptions} and \Cref{ass:convex_assumptions} hold. 
Consider a radius $\Delta > 0$ and the constants $\alpha, \beta>0$ from \eqref{eq:def_C_f_C_g_alpha} and \eqref{eq:defn_beta_C}. Also let again $s_k(\Delta)$ denote the corresponding global solution of the subproblem as in \eqref{eq:fcd_thm_subproblem}. For given $\tau \geq 0$ let $s_k \in H$ denote a point such that 
   \begin{align}
        \|s_k\| &\leq \Delta \notag\\
        m_k(0)-m_k(s_k)
        &\geq
        \left( \frac{1}{1+  \tau}\right)
        \left(
            m_k(0)-m_k(s_k(\Delta)) \label{eq:small_radi_eq1}
        \right).
\end{align}
Also define the constant
\begin{equation}
    \beta_{ \tau}
    :=
    \min\left\{
        \frac{D}{F_0^\alpha},
        \left(
            \frac{(1-\eta_2)}{(1+ \tau)2C_f D}
        \right)^\alpha
    \right\}.
    \label{eq:def_beta_tau}
\end{equation}
Then, if $\Delta \leq \beta_{ \tau} F_k^\alpha$ the corresponding trial is automatically very successful, i.e.
\begin{equation*}
\rho_k = \frac{f(x_k)-f(x_k+s_k)}{m_k(0)-m_k(s_k)}
  \geq \eta_2.
\end{equation*}
\end{lemma}
\begin{proof}
Let $\Delta\leq \beta_{ \tau} F_k^\alpha$. As $\beta_{\tau} \leq \beta$, from \Cref{lem:convex_model_decrease_alpha} we obtain 
\begin{equation*}
 m_k(0) - m_k(s_k(\Delta)) \overset{\eqref{eq:md_alpha},\beta_{ \tau} \leq \beta }{\geq} \frac{(1+\eta_2)F_k}{2D}\min\left \{\beta_{ \tau} F_k^{\alpha},\,\Delta \right \} 
 {=}
 \frac{(1+\eta_2)F_k \Delta}{2D} \geq \frac{F_k \Delta}{2D} .
\end{equation*}
and consequently using the assumptions of the theorem
\begin{equation}
    m_k(0) - m_k(s_k)\overset{\eqref{eq:small_radi_eq1}}{\geq}\left( \frac{1}{1+  \tau}\right)(m_k(0)-m_k(s_k(\Delta))) 
    \geq \frac{F_k \Delta}{(1+\tau)2D} .\label{eq:tr_small_radi_eq2} 
\end{equation}
By \Cref{lem:model_approximation}, we infer
\begin{equation}
|f(x_k+s_k)-m_k(s_k)|
\le C_f \Delta^{2+\nu}.\label{eq:tr_small_radi_eq3}
\end{equation}
Therefore, combining the previous bounds we obtain
\begin{equation*}
|1-\rho_k|
=
\frac{|f(x_k+s_k)-m_k(s_k)|}{ m_k(0) - m_k(s_k)}
\overset{\eqref{eq:tr_small_radi_eq2},\eqref{eq:tr_small_radi_eq3}}{\leq}
\frac{(1+ \tau)C_f \Delta^{2+\nu}}{(F_k \Delta)/(2D)}
=
2 (1+ \tau)C_f D\,\frac{\Delta^{1+\nu}}{F_k}.
\end{equation*}
Using $\Delta \leq \beta_{ \tau} F_k^\alpha$ and $\alpha(1+\nu)=1$ it follows that
\begin{equation*}
|1-\rho_k|
\leq
2  (1+ \tau) C_f D\,\beta_{ \tau}^{1+\nu}
\leq 1-\eta_2,
\end{equation*}
by the definition of $\beta_{ \tau}$. This proves the statement.
\end{proof}

We now obtain the main complexity result in the convex regime by combining \Cref{thm:main_abstract_convex_result} with \Cref{lem:small_radi_accepted_ver1}.

\begin{theorem}[Universal complexity of
\Crefalgvariant{var:variant1}{} under convexity]\label{thm:tr_convex_complexity}
Let \Cref{ass:nonconvex_assumptions} and \Cref{ass:convex_assumptions} hold, and let \((x_k)_k\) be generated by
\Crefalgvariant{var:variant1}{} of \Cref{alg:abstract_tr_algorithm}. Then
\[
    k_{\mathrm{TR},\mathrm{cvx}}(\varepsilon)
    :=
    \min\{k\in\mathbb N:f(x_{k})-f_\star\leq\varepsilon\}
    =
    \mathcal O\bigl(\varepsilon^{-1/(1+\nu)}\bigr)
    \qquad (\varepsilon\downarrow0).
\]
\end{theorem}

\begin{proof}
It suffices to apply \Cref{thm:main_abstract_convex_result} with
\[
\kappa_{\mathrm{cvx}} = \frac{1}{1 + \bar\tau},
\qquad
\beta_{\mathrm{cvx}} = \beta_{\bar \tau},
\]
where $\beta_{\bar \tau} \leq \beta$ is defined in \eqref{eq:def_beta_tau}, and where $\bar\tau$ denotes the upper bound on the sequence of accuracy parameters, that is, $\tau_k \leq \bar\tau$. The claim then follows immediately from \Cref{lem:small_radi_accepted_ver1}, applied with $\tau = \bar \tau$.

\end{proof}

\subsection{Fast local convergence with universal rates}
To complete the picture, we further analyze local fast convergence under local strong convexity of the objective. Although the proofs in this section are rather standard, we provide them to ensure the presentation remains self-contained. 
During the following section we again denote by $s_k^\star$ an arbitrary global solution of the subproblem, i.e.
\begin{equation*}
    s_k^\star \in
        \argmin_{\|s\|\leq \Delta_k} m_k(s).
\end{equation*}
\begin{theorem}[Local fast convergence, \Crefalgvariant{var:variant1}{}]\label{thm:tr_local_fast_convergence}
Let \Cref{ass:nonconvex_assumptions} hold. Let $x_\star$ be a local minimizer of $f$, and suppose that $\nabla^2 f(x_\star)\succ0$ and $x_k\to x_\star$. Then there exist an index $K_0 \in \mathbb{N}$ and a lower bound $\underline\Delta>0$ such that 
\begin{itemize}
    \item[(i)] for all $k\geq K_0$ the step is very successful, i.e. $\rho_k \geq \eta_2$.
    \item[(ii)] for all $k \geq K_0$, we have $\Delta_k\geq \underline\Delta$.
    \item[(iii)] for all \(k\geq K_0\), the exact trust-region solution satisfies
\[
    s_k^\star=-\nabla^2 f(x_k)^{-1}\nabla f(x_k),
\]
and both the computed step and the exact step are strictly feasible, i.e.
\[
    \max\{\|s_k\|,\|s_k^\star\|\}<\Delta_k.
\]

    \item[(iv)] If $\nu =0$ and $\tau_k \leq  \|\nabla f(x_k) \|^2$ the
     trust-region method eventually converges $Q$-superlinearly, i.e.
\begin{equation*}
    \|x_{k+1}-x_\star\|=o(\|x_k-x_\star\|).
\end{equation*}
\item[(v)]  If \(\nu\in(0,1]\) and $\tau_k \leq  \|\nabla f(x_k) \|^2$, then there exist a constant $C>0$ such that, for all sufficiently large $k$
\begin{equation*}
    \|x_{k+1}-x_\star\|
    \leq
    C\|x_k-x_\star\|^{1+\nu}.
\end{equation*}
\end{itemize}
\end{theorem}
\begin{proof}
    Since $\nabla^2 f(x_\star)\succ0$ and $x_k\to x_\star$, there exist by continuity of $\nabla^2f$ a radius $r>0$ and a pair $M \geq\mu>0$ as well as $K_0 \geq 0$ with
\begin{equation}
    x_k \in \overline{B_r(x^*)}, \qquad \sup_{x \in \overline{B_r(x_\star)}} \|\nabla^2 f(x)\|_{\mathrm{op}} \leq M, \qquad  H_k= \nabla^2 f(x_k)\succeq \mu I
    \qquad \forall k\ge K_0 . \label{eq:thm_local_eq_TR0}
\end{equation}
 First note that the exact solution $s_k^\star$ of the subproblem is unique because of \eqref{eq:thm_local_eq_TR0}. Also by mean value theorem 
 \begin{equation}
     \|\nabla f(x) - \nabla f(y)\| \leq M\|x- y\| \qquad x,y \in \overline{B_r(x_\star)} .\label{eq:thm_local_eq_TR01}
 \end{equation}
 The optimality conditions for the trust-region subproblem in \Cref{thm:global_solution_subproblem} imply that there exists a Lagrange multiplier $\lambda_k\geq 0 $ such that
\begin{equation}
       (H_k+\lambda_k I)s_k^\star=-g_k,
    \qquad
    H_k+\lambda_k I\succeq0,
    \qquad
    \lambda_k(\Delta_k-\|s_k^\star\|)=0. \label{eq:thm_local_eqTR1}
\end{equation}
Taking the inner product of the first equation with $s_k^\star$ and rearranging we have
\begin{equation}
    \|s_k^\star\|
    \leq
    \frac{1}{\mu + \lambda_k}\|\nabla f(x_k)\| 
    \leq
    \frac{1}{\mu}\|\nabla f(x_k)\| \qquad  \text{for $k \geq K_0$}.
    \label{eq:thm_local_eqTR2}
\end{equation}
Now we aim to show that $\|s_k\| \lesssim \| \nabla f(x_k)\|$ also in the inexact case. For this we distinguish two cases:\\[0.2cm]
\emph{Case 1: $\|s_k^\star\| < \Delta_k$}. Using the complementarity system in \Cref{thm:global_solution_subproblem} we infer $\lambda_k = 0$ and calculate
\begin{equation}
    m_k(0) - m_k(s_k^\star) = -\langle \nabla f(x_k),s_k^\star \rangle - \frac{1}{2}\langle H_k s_k^\star, s_k^\star \rangle \overset{\eqref{eq:thm_local_eqTR1}}{=} \frac{1}{2}\langle H_k s_k^\star, s_k^\star \rangle
    =
    \frac{1}{2}\langle H_k^{-1} \nabla f(x_k), \nabla f(x_k) \rangle
    \label{eq:thm_local_eq3}
\end{equation}
Moreover by $\eqref{eq:thm_local_eq_TR0}$, the model-objective $m_k$ is $\mu$--strongly convex for $k \geq K_0$, which implies
\begin{align*}
\begin{aligned}
    \frac{\mu}{2}\|s_k - s_k^\star \|^2 
    &\overset{\eqref{eq:thm_local_eq_TR0}}{\leq} m_k(s_k) - m_k(s_k^\star) =  [m_k(s_k) - m_k(0)] + [m_k(0) - m_k(s_k^\star) ] \\
    &\overset{\eqref{eq:fractional_model_descent}}{\leq} 
    -\left(\frac{1}{1 +\tau_k}\right)[m_k(0) - m_k(s_k^\star) ] + [m_k(0) - m_k(s_k^\star) ] \\
    &= \left(\frac{\tau_k}{1 +\tau_k}\right)[m_k(0) - m_k(s_k^\star) ] \\
    &\overset{\eqref{eq:thm_local_eq3}}{=} \left(\frac{\tau_k}{2(1 +\tau_k)}\right)\langle H_k^{-1} \nabla f(x_k), \nabla f(x_k) \rangle 
    \leq  
    \left(\frac{\tau_k}{2 \mu (1 +\tau_k)}\right) \| \nabla f(x_k)\|^2.
    \end{aligned}
\end{align*}
Consequently we deduce
\begin{equation}
    \|s_k - s_k^\star \| \leq \frac{\sqrt{\tau_k}} {\mu}\|\nabla f(x_k)\|.   \label{eq:thm_local_eqTR4}
\end{equation}
Combined with the upper bound $\tau_k \leq \bar \tau$, we infer
\begin{equation}
    \|s_k\| \leq \|s_k - s_k^\star \| + \| s_k^\star \| \overset{\eqref{eq:thm_local_eqTR2},\eqref{eq:thm_local_eqTR4}}{\leq} \left( 
    \frac{1+ \sqrt{\bar \tau} }{\mu} \right) \| \nabla f(x_k)\|.\label{eq:thm_local_eqTR5}
\end{equation}
\emph{Case 2: $\|s_k^\star\| = \Delta_k$}. In this case we just use \eqref{eq:thm_local_eqTR2} to infer
\begin{equation}
    \|s_k\| \leq \Delta_k = \|s_k^\star\| \overset{\eqref{eq:thm_local_eqTR2}}{\leq } \frac{1}{\mu} \| \nabla f(x_k)\| \label{eq:thm_local_eqTR6}
\end{equation}
Take now the Cauchy-decrease direction $s_k^C$ defined in \Cref{lem:normal_cauchy_decrease}, and infer using \eqref{eq:fractional_model_descent} that 
\begin{equation}
    m_k(0) - m_k(s_k) \geq \left( \frac{1}{1+\bar \tau}\right)\left( m_k(0) - m_k(s_k^\star)\right) \geq  \left( \frac{1}{1+\bar \tau}\right)\left(m_k(0) - m_k(s_k^C)\right). \label{eq:thm_local_eqTR55}
\end{equation}
Now we use Cauchy decrease condition in \Cref{lem:normal_cauchy_decrease}, and the inequalities above in \eqref{eq:thm_local_eqTR5},\eqref{eq:thm_local_eqTR6} to estimate from below
\begin{align*}
m_k(0) - m_k(s_k) 
\overset{\eqref{eq:thm_local_eqTR55}}{\geq} 
\left( \frac{1}{1+\bar \tau} \right)(m_k(0) - m_k(s_k^C)) 
&\overset{\eqref{eq:standard_cauchy_decrease}}{\geq}
\left( \frac{\kappa_C}{1+\bar \tau} \right) \,\|g_k\| \,
\min\!\left\{
\frac{\|g_k\|}{1+\|H_k\|_{\mathrm{op}}},\, \Delta_k
\right\} \\
&\overset{\eqref{eq:thm_local_eqTR5},\eqref{eq:thm_local_eqTR6}}{\geq}
\left( \frac{\kappa_C \mu}{(1+\sqrt{\bar \tau})(1+\bar \tau)}\right) \min\left\{ \frac{\mu}{(1+\sqrt{\bar \tau})(1+M)},1  \right \} \|s_k\|^2,
\end{align*}
where we also used $\|s_k\| \leq \Delta_k$ in the last inequality and the constant $M>0$ defined in \eqref{eq:thm_local_eq_TR0}. Let us denote the constant above in front of $\|s_k\|^2$ by $C>0$. Since $x_k\to x_\star$ and $\|s_k\|\to 0$, after possibly
enlarging $K_0$, we may assume that even
\[
    x_k+t s_k\in\overline{B_r(x^\star)}
    \qquad
    \text{for all }t\in[0,1]\text{ and }k\geq K_0.
\]
Since $H$ is finite-dimensional and
$\nabla^2f$ is continuous, it is uniformly continuous on
$\overline {B_r(x_\star)}$. Consequently, standard estimates using the integral Taylor formulas yield
\[
    m_k(s_k)-f(x_k+s_k)=o(\|s_k\|^2).
\]
Hence, we deduce
\begin{equation*}
|1-\rho_k|
    =
    \frac{|f(x_k+s_k)-m_k(s_k)|}
         {m_k(0)-m_k(s_k)}  \leq
         \frac{|f(x_k+s_k)-m_k(s_k)|}
         {C\|s_k\|^2}
    \to 0  \qquad \text{(as $k \to \infty$)}.
\end{equation*}
Thus we may assume after potentially enlarging $K_0$ that $\rho_k\ge\eta_2$ for all $k\geq K_0$. Hence from the radius update mechanism of \Cref{alg:abstract_tr_algorithm}, we infer
\begin{equation*}
    \Delta_k \geq \underline \Delta:= \Delta_{K_0} \qquad \text{for all $k \geq K_0$}. 
\end{equation*}
By \eqref{eq:thm_local_eqTR2}, \eqref{eq:thm_local_eqTR5} and \eqref{eq:thm_local_eqTR6} and since $\nabla f(x_k) \to 0$ as $k \to \infty$,  we may potentially again enlarge $K_0$ until 
\begin{align*}
    \max \{\|s_k \| ,\|s_k^\star \|\} < \underline \Delta \leq \Delta_k \qquad \text{for all $k \geq K_0$}. 
\end{align*}
Therefore using the complementarity conditions in \eqref{eq:thm_local_eqTR1} for the exact solution, we infer that the Lagrange multiplier for $s_k^\star$ satisfies $\lambda_k = 0$ for $k \geq K_0$. Consequently we deduce that
\begin{equation}
     s_k^\star=-\nabla^2 f(x_k)^{-1}\nabla f(x_k) \qquad \text{for all $k \geq K_0$},\label{eq:thm_local_eqTR7}
\end{equation}
which is the exact Newton step. It remains to prove $(iv)$ and $(v)$. Assume that $\tau_k \leq \|\nabla f(x_k)\|^2$ for all $k \geq K_0$. We set $x_{k+1}^\star:= x_k + s_k^\star$. By \eqref{eq:thm_local_eqTR7}, $x_{k+1}^\star$ is the exact Newton-step for $k \geq K_0$. Also note that $x_{k+1} - x_{k+1}^\star = s_k - s_k^\star$ and since  $\|s_k^\star\| < \Delta_k$ for $ k \geq K_0$, we deduce from \eqref{eq:thm_local_eqTR4} and $\nabla f(x_\star) = 0$ that 
\begin{align}
    \|x_{k+1} - x_\star\| \leq 
     \|x_{k+1} - x_{k+1}^\star\| + \|x_{k+1}^\star -  x_\star\| 
&\overset{\eqref{eq:thm_local_eqTR4}}{\leq} \frac{\sqrt{\tau_k}}{\mu}\| \nabla f(x_k) - \nabla f(x_\star) \|+ \|x_{k+1}^\star -  x_\star\| \notag \\
&\leq \frac{1}{\mu}\| \nabla f(x_k) - \nabla f(x_\star) \|^2+ \|x_{k+1}^\star -  x_\star\| \notag \\
&\overset{\eqref{eq:thm_local_eq_TR01}}{\leq} 
\frac{M^2}{\mu}\| x_k - x_\star \|^2+ \|x_{k+1}^\star -  x_\star\|. \label{eq:thm_local_eqTR8}
\end{align}
Since $x^\star_{k+1}$ is the classical Newton update, we obtain 
\begin{equation}
\begin{cases}
  \|x_{k+1}^\star -  x_\star\| = o(\|x_{k} -  x_\star\|) &\text{if $f \in C^2$} \\
  \|x_{k+1}^\star -  x_\star\| = \mathcal{O}(\|x_k -  x_\star\|^{1+ \nu}) &\text{if $f \in C^{2,\nu}$}
\end{cases}, \notag
\end{equation}
by classical textbook Newton analysis, see for instance \cite[Exercise~5.7.1]{kelley1995iterative}.
Consequently, statement $(iv)$ and $(v)$ follow immediately from \eqref{eq:thm_local_eqTR8}. 
\end{proof}
\begin{remark}
The statement in $(iv)$ remains valid for merely $C^2$--functions $f$ and does not rely on the Hölder condition with exponent $\nu=0$. The proof remains the same.
\end{remark}
The following corollary is a direct consequence. Under strong convexity we do not need to require that $x_k \to x_\star$.
\begin{corollary}[Fast convergence under strong convexity, \Crefalgvariant{var:variant1}{}] Let \Cref{ass:convex_assumptions} hold and additionally assume that $f:H \to \mathbb{R}$ is $\mu$--strongly convex. Also denote by $x_\star$ the unique global minimizer. Then there exist an index $K_0 \in \mathbb{N}$ and a lower bound $\underline\Delta>0$ such that $(i)$--$(v)$ from \Cref{thm:tr_local_fast_convergence} hold true, without requiring $x_k \to x_\star$ a priori. 
\end{corollary}
\begin{proof}
    It suffices to show that under the assumptions $x_k \to x_\star$. By \Cref{thm:tr_convex_complexity}, we have \(f(x_k)-f_\star\to0\).
Strong convexity therefore gives
\[
    \frac{\mu}{2}\|x_k-x_\star\|^2
    \leq f(x_k)-f_\star\to0.
\]
 The rest of the proof is identical to the one above.   
\end{proof}

\begin{remark}
\Cref{thm:tr_local_fast_convergence} recovers the classical fast local behavior of Newton-type methods for the trust-region scheme. Under the vanishing inexactness condition
\[
    \tau_k \leq \|\nabla f(x_k)\|^2,
\]
the method converges Q-superlinearly whenever \(f\in C^{2}\). If, in addition, \(\nabla^2 f\) is locally \((L_\nu,\nu)\)-Hölder continuous with
\(\nu\in(0,1]\), then the convergence has order \(1+\nu\). Importantly, this local rate is obtained adaptively, without knowledge of the Hölder exponent \(\nu\in[0,1]\) or the corresponding Hölder constant $L_\nu$. For \(\nu\in(0,1]\), this exponent is stronger than the local exponent \(1+\xi\), where
\[
    \xi = \frac{\nu}{1+\nu},
\]
obtained for the gradient-regularized Newton method in \cite[Theorem~6.6]{doikov2024super}. Although the latter rate is stated in gradient norm, under the nonsingularity condition \(\nabla^2 f(x_\star)\succ0\) the gradient norm and the distance to the solution are locally equivalent up to constants. The endpoint
\(\nu=0\) is not covered by \cite[Theorem~6.6]{doikov2024super}; related superlinear convergence results in the semismooth setting have recently
been obtained in \cite{alphonse2026skip}.
\end{remark}

\section{Analysis of the exact UniCAT method} \label{sec:cat_convergence}
In this section, we study \Crefalgvariant{var:variant2}{} of \Cref{alg:abstract_tr_algorithm}. In contrast to the analysis of \Crefalgvariant{var:variant1}{} in the previous section, we are able to establish universality in all three regimes: the nonconvex, convex, and locally strongly convex regimes.

Moreover, these results do not rely on the additional uniform boundedness assumption on the Hessian stated in \Cref{ass:bounded_hessian}.

We emphasize that \cite{hamad2022consistently,hamad2024simple} study the nonconvex and locally strongly convex settings under the assumption of Lipschitz-continuous Hessians. In contrast, we treat the full Hölder scale $\nu \in [0,1]$ and additionally cover the convex regime. Consequently, the results presented in this section are of independent interest, including in the nonconvex case.

However, we recall that, unlike the results in \cite{hamad2022consistently,hamad2024simple}, our analysis assumes exact subproblem solves. This simplifies the proofs, but limits the practical applicability of the resulting method.

During this section we frequently use the notation
\[
    \psi_k := \min \{\| \nabla f(x_k) \| , \| \nabla f(x_{k}+s_k)\|\},
    \qquad
    F_k := f(x_k)-f_\star
    \qquad (k\in\mathbb N).
\]

Also note that the notation for $F_k$, also introduced in \eqref{eq:def_function_gap} is mostly used for the analysis of convex optimization algorithms. Since we assume that the solution set is not empty this quantity is also well defined in the nonconvex case. However it will not automatically shrink to zero here. It was also used in the analysis in \cite{hamad2022consistently}.

\subsection{First-order universal complexity in the nonconvex regime}\label{sec:cat_nonconvex_complexity}
In this subsection we prove first-order complexity in the nonconvex regime of \Crefalgvariant{var:variant2}{} of \Cref{alg:abstract_tr_algorithm}. 
The following lemma was similarly proved in \cite{hamad2024simple} for $\nu =1$, we present a simplified proof for exact subproblem solves and the full Hölder scale $\nu \in [0,1]$.

\begin{lemma}[Gradient bound for \Crefalgvariant{var:variant2}{} of \Cref{alg:abstract_tr_algorithm}]
\label{lem:cat_gradient_control}
Let \Cref{ass:nonconvex_assumptions} hold, and suppose that $\eta_2<1/(1+\vartheta)$.  Furthermore set
\begin{equation*}
    K_\vartheta
    :=
    \frac{
        C_g+\frac{2C_f}{1-\eta_2}
    }{
        1-a_\vartheta
    }, \qquad   a_\vartheta
    :=
    \frac{\eta_2\vartheta}{1-\eta_2} <1.
\end{equation*}
Then, for every iteration $k \in \mathbb N$,
\begin{equation}
    \|s_k\|<\Delta_k
    \quad\text{or}\quad
    \rho_k^\vartheta<\eta_2
    \quad\ \Longrightarrow \quad
    \|\nabla f(x_{k,+})\|
    \leq
    K_\vartheta\|s_k\|^{1+\nu},\label{eq:nonconvex_grad_bound_main}
\end{equation}
where we again used the notation $x_{k,+} = x_k + s_k$.
\end{lemma}

\begin{proof}
Since $s_k$ solves the trust-region subproblem exactly, there exists according to \Cref{thm:global_solution_subproblem} a Lagrange multiplier $\lambda_k\geq0$ such that
\begin{equation}
    (H_k+\lambda_k I)s_k=-g_k,
    \qquad
    H_k+\lambda_k I\succeq0,
    \qquad
    \lambda_k(\Delta_k-\|s_k\|)=0.\label{eq:cat_nonconvex_gradlemma1}
\end{equation}
Hence $g_k+H_ks_k=-\lambda_k s_k$. Moreover, using again $ p_k(s_k) =
    m_k(0)-m_k(s_k)$ we have
\begin{align*}
    p_k(s_k)  =
    -\langle g_k,s_k\rangle
    -\frac12\langle H_ks_k,s_k\rangle 
    =
    \frac12\langle H_ks_k,s_k\rangle
    +\lambda_k\|s_k\|^2
    =
    \frac12\langle (H_k+\lambda_k I)s_k,s_k\rangle
    +\frac12\lambda_k\|s_k\|^2.
\end{align*}
Since $H_k+\lambda_k I\succeq 0$, we obtain for the predicted reduction
\begin{equation}
    p_k(s_k)\geq \frac12\lambda_k \|s_k\|^2. \label{eq:cat_nonconvex_gradlemma11}
\end{equation}
Using $g_k+H_ks_k=-\lambda_k s_k$ from above and \eqref{eq:eq_model_approx_grad} from \Cref{lem:model_approximation} the triangle inequality yields
\begin{equation}
    \|\nabla f(x_{k,+})\| \leq \| \nabla f(x_{k,+}) -\nabla f(x_{k})- H_ks_k\| + \|\nabla f(x_{k}) + H_ks_k\|
    \leq
    \lambda_k \|s_k\| + C_g \|s_k\|^{1+\nu}.\label{eq:cat_nonconvex_gradlemma2}
\end{equation}
We now distinguish two cases.\\[0.2cm]
\emph{Case 1:} First suppose $\|s_k\|<\Delta_k$. Then the complementarity condition in \eqref{eq:cat_nonconvex_gradlemma1} gives $\lambda_k=0$. Hence, by \eqref{eq:cat_nonconvex_gradlemma2}
\begin{equation*}
    \|\nabla f(x_{k,+})\| \leq
 C_g \|s_k\|^{1+\nu} \leq K_{\vartheta} \|s_k\|^{1+\nu} 
\end{equation*}
\emph{Case 2}: Now suppose $\rho_k^\vartheta<\eta_2$. First note that if $\|s_k\| = 0$, then $x_{k,+} = x_k$ and exact optimality
implies $g_k = 0$ so the claim is trivial. We may therefore assume $\|s_k\| >0 $.  The assumption $\rho_k^\vartheta<\eta_2$ yields
\begin{equation*}
    f(x_k) - f(x_k+s_k)
    <
    \eta_2
    \left(
        p_k(s_k)+\frac{\vartheta}{2} \psi_k \|s_k\|
    \right).
\end{equation*}
On the other hand, \eqref{eq:eq_model_approx_func} in \Cref{lem:model_approximation}  yields
\begin{equation*}
     f(x_k) - f(x_k+s_k) \geq p_k(s_k) - C_f\|s_k\|^{2+\nu}.
\end{equation*}
Combining the latter two estimates, also using $\psi_k \leq \|\nabla f(x_{k,+})\|$ we obtain
\begin{equation*}
p_k(s_k)-C_f \|s_k\|^{2+\nu}
    <
    \eta_2 p_k(s_k)
    +
    \frac{\eta_2\vartheta}{2}\psi_k \|s_k\| \leq 
    \eta_2 p_k(s_k)
    +
    \frac{\eta_2\vartheta}{2}\|\nabla f(x_{k,+})\|\|s_k\|
\end{equation*}
Using also $p_k(s_k) \geq \frac12\lambda_k \|s_k\| ^2$ from \eqref{eq:cat_nonconvex_gradlemma11} and rearranging, we obtain
\begin{equation*}
    \left(\frac{1-\eta_2}{2}\right)\lambda_k \|s_k\|^2
    \leq
    C_f \|s_k\|^{2+\nu}
    +
    \frac{\eta_2\vartheta}{2}\|\nabla f(x_{k,+})\|\|s_k\|.
\end{equation*}
Dividing by $\|s_k\|>0$, we get
\begin{equation*}
    \lambda_k \|s_k\|
    \leq
    \frac{2C_f}{1-\eta_2}\|s_k\|^{1+\nu}
    +
    \frac{\eta_2\vartheta}{1-\eta_2}\|\nabla f(x_{k,+})\|.
\end{equation*}
Combining the latter inequality and \eqref{eq:cat_nonconvex_gradlemma2} gives
\begin{equation*}
    \|\nabla f(x_{k,+})\|
    \leq
    \left(
        C_g+\frac{2C_f}{1-\eta_2}
    \right)\|s_k\|^{1+\nu}
    +
    \frac{\eta_2\vartheta}{1-\eta_2}\|\nabla f(x_{k,+})\|= 
     \left(
        C_g+\frac{2C_f}{1-\eta_2}
    \right)\|s_k\|^{1+\nu}
    +
    a_\vartheta\|\nabla f(x_{k,+})\|
\end{equation*}
Since $a_\vartheta \in (0,1)$ by construction, rearranging yields  \eqref{eq:nonconvex_grad_bound_main}.
\end{proof}
The following lemma is also inspired by \cite[Lemma~2]{hamad2024simple} and the corresponding conference version \cite{hamad2022consistently}.

\begin{lemma}[Radius lower bound for \Crefalgvariant{var:variant2}{} of \Cref{alg:abstract_tr_algorithm}]
\label{lem:cat_nonconvex_radius_lower_bound}
Let the assumptions of \Cref{lem:cat_gradient_control} hold, and fix
$\varepsilon>0$. Define
\begin{equation*}
    d_\varepsilon
    :=
    \left(
        \frac{\varepsilon}{K_\vartheta}
    \right)^{\frac{1}{1+\nu}},
    \qquad
    \delta_\varepsilon
    :=
    \min\{\Delta_0,\gamma_1d_\varepsilon\}.
\label{eq:ncvx_cat_lower_bound_radii1}
\end{equation*}
Suppose that
\begin{equation}
   \psi_n=  \min \{ \|\nabla f(x_n)\|,\|\nabla f(x_n+s_n)\| \}
    \geq \varepsilon
    \qquad \text{for } n=0,\ldots,k 
    \label{eq:cat_radi_lower_bound}
\end{equation}
Then the following two inequalities hold
\begin{equation}
    \Delta_n \geq \delta_\varepsilon
    \qquad \text{for } n=0,\ldots,k+1,
\label{eq:ncvx_cat_lower_bound_radii2}
\end{equation}
and
\begin{equation}
    \|s_n\|\geq \delta_\varepsilon
    \qquad \text{for } n=0,\ldots,k.    \label{eq:ncvx_cat_lower_bound_radii3}
\end{equation}
\end{lemma}

\begin{proof}
Suppose that \eqref{eq:cat_radi_lower_bound} holds. We first prove the lower bound \eqref{eq:ncvx_cat_lower_bound_radii2}. By definition of $\delta_\varepsilon$, we have
$\delta_\varepsilon\leq \Delta_0$. Starting from this, we aim to show that, for every $n=0,\ldots,k$,
\begin{equation}
    \Delta_n\geq\delta_\varepsilon
    \quad\Longrightarrow\quad
    \Delta_{n+1}\geq\delta_\varepsilon.
     \label{eq:cat_radi_lower_bound_eq1}
\end{equation}
Together with \(\delta_\varepsilon\leq \Delta_0\) this will prove  \eqref{eq:ncvx_cat_lower_bound_radii2}. To show \eqref{eq:cat_radi_lower_bound_eq1}, fix $n\in\{0,\ldots,k\}$ and assume that $\Delta_n \geq \delta_{\varepsilon}$. We distinguish two cases. \\[0.2cm]
\emph{Case 1:} $\rho_n^\vartheta < \eta_2$. By assumption, we have $\varepsilon\leq\psi_n\leq\|\nabla f(x_n+s_n)\| $ and since $\rho_n^\vartheta<\eta_2$ \Cref{lem:cat_gradient_control} applies and gives
\begin{equation}
    \varepsilon
    \leq
    \|\nabla f(x_n+s_n)\|
    \leq
    K_\vartheta\|s_n\|^{1+\nu}. \label{eq:cat_radi_lower_bound_eq2}
\end{equation}
after rearranging this gives
\[
    \|s_n\|\geq d_\varepsilon.
\]
Using feasibility, $\|s_n\| \leq \Delta_n$, and the radius update mechanism, we obtain
\[
    \Delta_{n+1}
    \geq
    \gamma_1\Delta_n
    \geq
    \gamma_1\|s_n\|
    \geq
    \gamma_1 d_\varepsilon
    \geq
    \delta_\varepsilon.
\]
\emph{Case 2:} If $\rho_n^\vartheta \geq \eta_2$, then the radius update mechanism gives
\[
    \Delta_{n+1}\geq \Delta_n\geq\delta_\varepsilon.
\]
This proves \eqref{eq:cat_radi_lower_bound_eq1} and consequently also 
\[
    \Delta_n\geq\delta_\varepsilon
    \qquad n=0,\ldots,k+1.
\]
It remains to prove the lower bound \eqref{eq:ncvx_cat_lower_bound_radii3}. For that purpose fix $n\in\{0,\ldots,k\}$. We distinguish three cases.\\[0.2cm]
\emph{Case 1:} $\rho_n^\vartheta < \eta_2$. In this case the same argument as in  \eqref{eq:cat_radi_lower_bound_eq2} gives
\[
    \|s_n\|\geq d_\varepsilon \geq \delta_\varepsilon.
\]
\emph{Case 2:} $\rho_n^\vartheta\geq\eta_2$ and $\|s_n\|<\Delta_n$. In this case again \Cref{lem:cat_gradient_control} applies and the same argument as in \eqref{eq:cat_radi_lower_bound_eq2} gives
\[
    \varepsilon
    \leq \psi_n \leq 
    \|\nabla f(x_n+s_n)\|
    \leq
    K_\vartheta\|s_n\|^{1+\nu},
\]
and hence
\[
    \|s_n\|\geq d_\varepsilon\geq\delta_\varepsilon.
\]
\emph{Case 3:} $\rho_n^\vartheta\geq\eta_2$ and $\|s_n\|=\Delta_n$. In this case the radius lower bound \eqref{eq:ncvx_cat_lower_bound_radii2} gives
\[
    \|s_n\|=\Delta_n\geq\delta_\varepsilon.
\]
Thus combining all cases we obtain
\[
    \|s_n\|\geq\delta_\varepsilon
    \qquad n=0,\ldots,k,
\]
which proves the claim.
\end{proof}

Now we are in the position to prove the first-order iteration complexity bound for the UniCAT method.

\begin{theorem}[First-order universal complexity of
\Crefalgvariant{var:variant2}{}]
\label{thm:cat_nonconvex_complexity}
Let \Cref{ass:nonconvex_assumptions} hold, and suppose that the trust-region subproblem is solved exactly at every iteration. Moreover assume $\eta_2<\frac{1}{1+\vartheta}$ and define quantity
\begin{equation*}
    k_{\mathrm{CAT,ncvx}}(\varepsilon):= \min \{k \in \mathbb N: \min\{ \|\nabla f(x_k) \|, \|\nabla f(x_k + s_k) \|\} \leq \varepsilon \} .
\end{equation*}
Then the following complexity bound holds for $\varepsilon \leq (\Delta_0/\gamma_1)^{1+\nu} K_\vartheta$:
 \begin{align*}
   k_{\mathrm{CAT,ncvx}}(\varepsilon)
    \leq
    \left(
        1+\frac{\log(\gamma_3)}{|\log(\gamma_2)|}
    \right) \left(
    \frac{C_{\mathrm{CAT}}}{\varepsilon^{(2+\nu)/(1+\nu)}}\right)
    +
    \frac{1}{|\log(\gamma_2)|}
     \log_+\left( 
        \frac{\Delta_0 K_\vartheta^\alpha }{\gamma_1\varepsilon^\alpha}
    \right),
    \qquad C_{\mathrm{CAT}} := \frac{2F_0 K_\vartheta^\alpha}{\eta_1 \gamma_1 \vartheta}. 
\end{align*}
Hence we have $ k_{\mathrm{CAT,ncvx}}(\varepsilon) = \mathcal{O}(\varepsilon^{-(2+\nu)/(1+\nu)})$  as $\varepsilon \downarrow 0$.
\end{theorem}

\begin{proof}
We first show that $ k_{\mathrm{CAT,ncvx}}(\varepsilon)<\infty$ for arbitrary $\varepsilon>0$. Suppose, to the contrary, that this is not the case. Then $\psi_n>\varepsilon$ for all $n\in\mathbb N $. Consequently, \eqref{eq:ncvx_cat_lower_bound_radii3} yields $\|s_n \| \geq \delta_\varepsilon >0$ for all $n\in \mathbb N$, which in turn implies
\begin{align}
    f(x_{n}) - f(x_{n+1}) \overset{(a)}{\geq} \eta_1 \left( m_n(0) - m_n(s_n) + \frac{\vartheta}{2} \psi_n \|s_n\|\right)
    \overset{(b)}{\geq}
     \frac{ \eta_1 \vartheta \psi_n \|s_n\|}{2} \overset{\eqref{eq:ncvx_cat_lower_bound_radii3}}{\geq} \frac{\eta_1 \vartheta}{2}\varepsilon \delta_\varepsilon,\quad \text{$(n \in \mathcal{S}$)} 
     \label{eq:cat_nonconvex_eq0}.
\end{align}
Here we used in $(a)$ that $n \in \mathcal{S}$ and consequently $\rho_n^\vartheta \geq \eta_1$ and in $(b)$ that $s_n$ is a global minimizer of the model-problem.
By the standing convention in \Cref{par:standing_convention}, we have $|\mathcal{S}|=\infty$. Together with \Cref{lem:well-definedness}, this implies that
\begin{equation*}
f(x_n)-f(x_{n+1})\to 0
\qquad\text{as } n\to\infty,\quad n\in\mathcal{S}.    
\end{equation*}
This contradicts the strictly positive lower bound in \eqref{eq:cat_nonconvex_eq0}.
\\
We now prove the complexity bound. For that purpose, fix $0<\varepsilon \leq (\Delta_0/\gamma_1)^{1+\nu} K_\vartheta$ and note that if $k_{\mathrm{CAT,ncvx}}(\varepsilon)=0$ the theorem is trivially satisfied as the right hand side is positive for any $\varepsilon>0$. So we assume  $k_{\mathrm{CAT,ncvx}}(\varepsilon) \geq1$ and follow the arguments in \Cref{lem:successful_rate_abstract} and set
\begin{equation*}
K(\varepsilon):= k_{\mathrm{CAT,ncvx}}(\varepsilon),\qquad 
    \mathcal{S}(\varepsilon):= 
    \mathcal{S}_{k_{\mathrm{CAT,ncvx}}(\varepsilon)-1}. 
\end{equation*}
In order to apply \Cref{thm:relation_successful_unsuccessful}, we have to bound the cardinality of the latter set from above. For this purpose consider for $n \in \mathcal{S}(\varepsilon)$ the estimate as in \eqref{eq:cat_nonconvex_eq0}:
\begin{align}
    F_n - F_{n+1} = f(x_{n}) - f(x_{n+1}) \overset{}{\geq} \eta_1 \left( m_n(0) - m_n(s_n) + \frac{\vartheta}{2} \psi_n \|s_n\|\right)
    \overset{}{\geq}
     \frac{ \eta_1 \vartheta \psi_n \|s_n\|}{2} \overset{\eqref{eq:ncvx_cat_lower_bound_radii3}}{\geq} \frac{\eta_1 \vartheta}{2}\varepsilon \delta_\varepsilon,\quad 
     \label{eq:cat_nonconvex_eq1}
\end{align}
where we now in the last inequality used \Cref{lem:cat_nonconvex_radius_lower_bound} which applies as $\psi_n \geq \varepsilon$ for $n \in \mathcal{S}(\varepsilon)$. Summing the latter from $n=0$ to $n = K(\varepsilon) -1 $, we obtain
\begin{equation*}
F_{0} - F_{K(\varepsilon)} = 
    \sum_{n = 0}^{K(\varepsilon)-1}\left(F_{n} - F_{n+1} \right) \overset{(a)}{=} \sum_{n \in  \mathcal{S}(\varepsilon)}\left(F_{n} - F_{n+1} \right) \overset{\eqref{eq:cat_nonconvex_eq1}}{\geq}  \frac{\eta_1 \vartheta}{2}\varepsilon \delta_\varepsilon |\mathcal{S}(\varepsilon)|,\label{eq:cat_compl_sum}
\end{equation*}
where we used in $(a)$ that $F_n = F_{n+1}$ for $n \notin \mathcal{S}$. Rearranging the previous inequality yields
\begin{equation}
    |\mathcal{S}(\varepsilon)| \leq \frac{ 2 F_0}{\eta_1 \vartheta \varepsilon \delta_\varepsilon} = \frac{ 2 K_\vartheta^{1/(1+\nu)} F_0}{\eta_1 \vartheta \gamma_1 \varepsilon^{(2+\nu)/(1+\nu)}}, \label{eq:cat_upperbound_successful}
\end{equation}
where we used that $ \delta_\varepsilon = \gamma_1(\varepsilon/K_\vartheta)^{1/(1+\nu)}$ for $\varepsilon \leq  (\Delta_0/\gamma_1)^{1+\nu} K_\vartheta$. By minimality of $K(\varepsilon)$, $\psi_n > \varepsilon$ for every $n = 0, \ldots , K(\varepsilon)-1$. Hence \Cref{lem:cat_nonconvex_radius_lower_bound} applies and we deduce
\begin{equation*}
\delta_\varepsilon
\overset{\eqref{eq:ncvx_cat_lower_bound_radii2}}{\leq} \Delta_{K(\varepsilon)}.
\end{equation*}
Now the global complexity bound follows from \Cref{thm:relation_successful_unsuccessful} with $k = K(\varepsilon)$ and $\Delta_\varepsilon := \delta_\varepsilon$ under usage of the upper bound \eqref{eq:cat_upperbound_successful}:
 \begin{align*}
    K(\varepsilon)
    \leq
    \left(
        1+\frac{\log(\gamma_3)}{|\log(\gamma_2)|}
    \right) 
    \frac{C_{\mathrm{CAT}}}{\varepsilon^{(2+\nu)/(1+\nu)}}
    +
    \frac{1}{|\log(\gamma_2)|}
     \log_+\left( 
        \frac{\Delta_0 K_\vartheta^\alpha }{\gamma_1\varepsilon^\alpha}
    \right),\qquad 
    C_{\mathrm{CAT}} = \frac{2F_0K_\vartheta^\alpha}{\eta_1 \gamma_1 \vartheta},
\end{align*}
which proves the claimed complexity bound.

\end{proof}

\begin{remark}[Stationarity measure in the nonconvex analysis of \Crefalgvariant{var:variant2}{}]
In the nonconvex analysis of \Crefalgvariant{var:variant2}{}, stationarity is measured by
\[
    \psi_k
    =
    \min \{\| \nabla f(x_k) \| , \| \nabla f(x_{k}+s_k)\|\},
\]
rather than by \(\|g_k\|=\|\nabla f(x_k)\|\), as in the classical trust-region complexity result in \Cref{thm:tr-nonconvex-complexity}. This does not change
the interpretation of the result as a first-order stationarity guarantee. Indeed,
at iteration \(k\), both \(x_k\) and the trial point \(x_k+s_k\) are computed points. Thus \Cref{thm:cat_nonconvex_complexity} bounds the number of iterations until an \(\varepsilon\)-first-order stationary point has been computed. The only distinction from the bound in \Cref{thm:tr-nonconvex-complexity} is that the returned point need not be the
accepted iterate \(x_{k+1}\); it may also be the trial point \(x_{k}+s_k\). 
\end{remark}

\subsection{Universal complexity in the convex regime}\label{sec:cat_convex_complexity}
We now turn to the second main contribution of this article, namely the universal convex iteration complexity analysis of \Crefalgvariant{var:variant2}{} of \Cref{alg:abstract_tr_algorithm}. To apply \Cref{thm:main_abstract_convex_result}, we again need to verify that \eqref{eq:cvx_abstract_thm_lower_bound_radii} holds. In other words, we have to show that, also in this setting, sufficiently small radii are eventually very successful.

\begin{lemma}[Small radii lead to very successful steps]\label{lem:cat_small_radii_accepted}
Let \Cref{ass:nonconvex_assumptions} and \Cref{ass:convex_assumptions} hold. Consider $k \in \mathbb{N}$, a radius $\Delta > 0$ and the constants $\alpha, \beta>0$ from \eqref{eq:def_C_f_C_g_alpha}  and  \eqref{eq:defn_beta_C}. Also let again $s_k =s_k(\Delta)$ denote the corresponding global solution of the subproblem as in \eqref{eq:fcd_thm_subproblem}. Moreover fix   $\vartheta > 0$ such that $0<  \eta_2<\frac{1}{1+\vartheta/2}$. Define the constants
\begin{equation}
    \beta_{\mathrm{CAT}}
    :=
    \min\left\{
        \beta_{},
        \left(\frac{R_\vartheta}{2D}\right)^{\frac{1}{1+\nu}}
    \right\}, \qquad  R_\vartheta
    :=
    \frac{
        1-\eta_2(1+\vartheta/2)
    }{
        C_f+\eta_2\vartheta C_g/2
    }>0. \label{eq:def_beta_cat}
\end{equation}
Then if $\Delta \leq \beta_{\mathrm{CAT}} F_k^\alpha$,
the corresponding trial is automatically very successful, i.e.
\begin{equation*}
\rho_k^\vartheta := \frac{f(x_k)-f(x_k+s_k)}{m_k(0)-m_k(s_k) + \frac{\vartheta}{2}\min\{ \|\nabla f(x_k)\|,\|\nabla f(x_{k,+})\|\}\|s_k\|}
  \geq \eta_2.
\end{equation*}
\end{lemma}

\begin{proof} The first part of the proof is almost identical to the one of \Cref{lem:cat_gradient_control}: Since $s_k$ solves the trust-region subproblem exactly, there again exists, by \Cref{thm:global_solution_subproblem} a Lagrange multiplier $\lambda_k\geq 0$ such that
\begin{equation}
    (H_k+\lambda_k I)s_k=-g_k,
    \qquad
    H_k+\lambda_k I\succeq0,
    \qquad
    \lambda_k(\Delta-\|s_k\|)=0. \label{eq:cat_lower_bound_cvx_1}
\end{equation}
Because $f$ is convex, $H_k=\nabla^2 f(x_k)\succeq0$ and we deduce 
\begin{equation*}
p_k(s_k) = -\langle g_k,s_k\rangle -\frac{1}{2}\langle H_ks_k,s_k\rangle 
 =
\frac{1}{2}\langle H_ks_k,s_k\rangle
+\lambda_k\|s_k\|^2 
\geq
\lambda_k \|s_k\|^2 .
\end{equation*}
Using $g_k+H_ks_k=-\lambda_k s_k$ and \eqref{eq:eq_model_approx_grad} in \Cref{lem:model_approximation} the triangle inequality yields
\begin{equation*}
    \|\nabla f(x_{k,+})\| \leq \| \nabla f(x_{k,+}) -\nabla f(x_{k})- H_ks_k\| + \|\nabla f(x_{k}) + H_ks_k\|
    \leq
    \lambda_k \|s_k\| + C_g \|s_k\|^{1+\nu}. \label{eq:thm_cat_accept_eq01}
\end{equation*}
Consequently, as $ \psi_k\leq \|\nabla f(x_{k,+})\|$
\begin{align*}
   \psi_k \|s_k\|
    \leq
    \|\nabla f(x_{k,+})\| \|s_k\|
    \leq
    \lambda_k \|s_k\|^2 + C_g \|s_k\|^{2+\nu}
    \leq
    p_k(s_k) + C_g \|s_k\|^{2+\nu}.
\end{align*}
Since $\|s_k\| \leq \Delta$ the latter gives 
\begin{equation}
    \psi_k \|s_k\|
    \leq
    p_k(s_k) + C_g\Delta^{2+\nu}. \label{eq:thm_cat_accept_eq1}
\end{equation}
Note that due to \eqref{eq:md_alpha}, the predicted decrease satisfies $p_k(s_k) > 0$ as $\Delta>0$ and $F_k>0$ by the standing convention in \Cref{par:standing_convention}. Using in addition \Cref{lem:model_approximation} we may compute
\begin{align}
    f(x_k)-f(x_k+s_k)
    = f(x_k) - m_k(s_k) + m_k(s_k) -f(x_k+s_k)
    &\geq
    p_k(s_k)-C_f \|s_k\|^{2+\nu} \notag \\
    &\geq
    p_k(s_k)-C_f\Delta^{2+\nu} \notag\\
    &= p_k(s_k)\left( 1 - \frac{C_f\Delta^{2+\nu}}{p_k(s_k)}
    \right). \label{eq:thm_cat_accept_eq2}
\end{align}
We will now bound $\Delta^{2+\nu}/{p_k(s_k)}$ from above. By assumption we have $ \Delta \leq \beta_{\mathrm{CAT}}F_k^\alpha$. Since $\beta_{\mathrm{CAT}}\leq \beta$ by construction, \eqref{eq:md_alpha} yields
\begin{equation*}
    p_k(s_k) =
    m_k(0)-m_k(s_k)
    \geq
    \frac{F_k\Delta}{2D}
\end{equation*}
and consequently
\begin{equation}
    \frac{\Delta^{2+\nu}}{p_k(s_k)}
    \leq
    2D\frac{\Delta^{1+\nu}}{F_k}
    \leq 
    2D\beta_{\mathrm{CAT}}^{1+\nu}
    \leq
    R_\vartheta, \label{eq:thm_cat_accept_eq21}
\end{equation}
where we used $\Delta \leq \beta_{\mathrm{CAT}}F_k^\alpha$ and $\alpha(1+\nu)=1$ and also the definition of $R_\vartheta$. We first conclude
\begin{equation}
    f(x_k)-f(x_k+s_k)
    \geq
    p_k(s_k)-C_f \Delta^{2+\nu} 
    \geq
     p_k(s_k)\left( 1 - {C_f R_\vartheta}
    \right) >0,  \label{eq:thm_cat_accept_eq22}
\end{equation}
where the last inequality follows from the definition of $R_\vartheta$. Additionally, the denominator of the ratio $\rho_k^\vartheta$, defined in \eqref{eq:rho_cat_def} satisfies
\begin{equation}
    (m_k(0)- m_k(s_k)) + \frac{\vartheta}{2}\psi_k \|s_k\|
    \overset{\eqref{eq:thm_cat_accept_eq1}}{\leq}
    \left(1+\frac{\vartheta}{2}\right)p_k(s_k)
    +
    \frac{\vartheta C_g}{2}\Delta^{2+\nu}. \label{eq:thm_cat_accept_eq3}
\end{equation}
Therefore, also using positivity of the second term in \eqref{eq:thm_cat_accept_eq22}, we obtain
\begin{equation*}
    \rho_k^\vartheta = \frac{f(x_k)-f(x_k+s_k)}{p_k(s_k) + \frac{\vartheta}{2}\psi_k \|s_k\|}
    \overset{\eqref{eq:thm_cat_accept_eq2},\eqref{eq:thm_cat_accept_eq3}}{\geq}
    \frac{
        p_k(s_k)-C_f\Delta^{2+\nu}
    }{
        \left(1+\frac{\vartheta}{2}\right)p_k(s_k)
        +
        \frac{\vartheta C_g}{2}\Delta^{2+\nu}
    }.
\end{equation*}
Dividing the numerator and denominator in the preceding lower bound by $p_k(s_k)>0$ gives
\begin{equation*}
\rho_k^\vartheta \geq \frac{
        1-C_f\frac{\Delta^{2+\nu}}{p_k(s_k)}
    }{
        1+\frac{\vartheta}{2}
        +
        \frac{\vartheta C_g}{2}
        \frac{\Delta^{2+\nu}}{p_k(s_k)} 
    } 
    \overset{\eqref{eq:thm_cat_accept_eq21}}{\geq}
     \frac{
        1-C_f R_\vartheta
    }{
        1+\frac{\vartheta}{2}
        +
        \frac{\vartheta C_g}{2}
        R_\vartheta
    } \geq \eta_2.  
\end{equation*}
Here we used the definition of $R_\vartheta$ in the last inequality. This proves the claim.
\end{proof}
As a simple corollary we now derive the global complexity bound for \Crefalgvariant{var:variant2}{} of \Cref{alg:abstract_tr_algorithm} in the convex regime.
\begin{theorem}[Universal complexity of \Crefalgvariant{var:variant2}{} under convexity]\label{thm:cat_convex_complexity} Let \Cref{ass:nonconvex_assumptions} and \Cref{ass:convex_assumptions} hold. Take $\varepsilon>0$ and let $(x_k)_k$ denote the iterates produced by  \Crefalgvariant{var:variant2}{} of \Cref{alg:abstract_tr_algorithm}.  Also fix   $\vartheta > 0$ and assume
\[
    0<  \eta_2<\frac{1}{1+\vartheta/2}.
\] Then we have 
\begin{equation*}
    k_{\mathrm{CAT},\mathrm{cvx}}(\varepsilon) := \min\{k \in \mathbb{N}: f(x_{k}) - f_\star \leq \varepsilon\} = \mathcal{O}(\varepsilon^{-\frac{1}{1+\nu}})\label{eq:def_k_eps_2},
\end{equation*}
as $\varepsilon \downarrow 0$.
\end{theorem}

\begin{proof}
For the proof it suffices to apply \Cref{thm:main_abstract_convex_result} with 
\[
\kappa_{\mathrm{cvx}} = 1, \quad
\beta_{\mathrm{cvx}} = \beta_{\mathrm{CAT}},
 \]
where $ \beta_{\mathrm{CAT}} \leq \beta$ is defined in \eqref{eq:def_beta_cat}. Then taking into account \Cref{lem:cat_small_radii_accepted}, the statement follows immediately.
\end{proof}

\subsection{Fast local convergence with universal rates}

As for \Crefalgvariant{var:variant1}{}, we finally study the local behavior of \Crefalgvariant{var:variant2}{} near a nondegenerate local minimizer. The proof follows the same strategy as in
\Cref{thm:tr_local_fast_convergence}.
\begin{theorem}[Local fast convergence, \Crefalgvariant{var:variant2}{}]\label{thm:tr_local_fast_convergence_cat}
Let \Cref{ass:nonconvex_assumptions} hold. Let $x_\star$ be a local minimizer of $f$, and suppose that $\nabla^2 f(x_\star)\succ0$ and $x_k\to x_\star$.  Moreover assume that
\[
\eta_2 <\frac{1}{1+\vartheta/2}.
\]
Then there exist an index $K_0 \in \mathbb{N}$ and a lower bound $\underline\Delta>0$ such that 
\begin{itemize}
    \item[(i)] for all $k\geq K_0$ the step is very successful, i.e. $\rho_k^\vartheta \geq \eta_2$.
    \item[(ii)] for all $k \geq K_0$, we have $\Delta_k\geq \underline\Delta$.
    \item[(iii)] for all $k \geq K_0$ the trust-region constraint is inactive in the sense that $\|s_k \|  < \Delta_k$ and we have 
    \begin{align*}
        s_k=-\nabla^2 f(x_k)^{-1}\nabla f(x_k), 
    \end{align*}
    which is the classical Newton step.
     \item[(iv)] If $\nu =0$,  \Crefalgvariant{var:variant2}{}  eventually converges $Q$-superlinearly, i.e.
\begin{equation*}
    \|x_{k+1}-x_\star\|=o(\|x_k-x_\star\|).
\end{equation*}
\item[(v)]  If \(\nu\in(0,1]\), then there exist a constant $C>0$ such that, for all sufficiently large $k$
\begin{equation*}
    \|x_{k+1}-x_\star\|
    \leq
    C\|x_k-x_\star\|^{1+\nu}.
\end{equation*}
\end{itemize}
\end{theorem}

\begin{proof}
   Since $\nabla^2 f(x_\star)\succ0$ and $x_k\to x_\star$, there exist by continuity of $\nabla^2f$ a $\mu>0$ and $K_0 \geq 0$ such that
\begin{equation}
    \nabla^2 f(x_k)\succeq \mu I
    \qquad \forall k\ge K_0 . \label{eq:thm_local_eq0}
\end{equation}
Recall that for \Crefalgvariant{var:variant2}{},  $s_k$ is the exact solution of the trust-region subproblem, which is unique because of \eqref{eq:thm_local_eq0}. The optimality conditions for the trust-region subproblem in \Cref{thm:global_solution_subproblem} imply that there exists a Lagrange multiplier $\lambda_k\geq 0 $
such that
\[
    (H_k+\lambda_k I)s_k=-g_k,
    \qquad
    H_k+\lambda_k I\succeq0,
    \qquad
    \lambda_k(\Delta_k-\|s_k\|)=0.
\]
In particular, $g_k+H_ks_k=-\lambda_k s_k$, which yields for $k\geq K_0$ the following estimate for the predicted reduction 
\begin{align}
    p_k(s_k)
    =
    -\langle g_k,s_k\rangle
    -\frac12\langle H_ks_k,s_k\rangle             
    =
    \frac12\langle H_ks_k,s_k\rangle
    +
    \lambda_k\|s_k\|^2                              
    \geq
    \frac{\mu}{2}\|s_k\|^2+\lambda_k \|s_k\|^2 . \label{eq:cat_local_1}
\end{align}
where we set again $p_k(s_k) := m_k(0)-m_k(s_k)$ for the predicted reduction. 
Also,
\begin{equation}
    \|s_k\|
    \le
    \|(H_k+\lambda_k I)^{-1}\|\,\|g_k\|
    \le
    \frac{1}{\mu}\|g_k\| \to 0, \qquad (k \to \infty )  \label{eq:cat_local_11}
\end{equation}
since \(x_k\to x_\star\) and  \(g_k\to0\). Using \(g_k+H_ks_k=-\lambda_k s_k\), we get as in \eqref{eq:thm_cat_accept_eq1}
\begin{equation*}
     \psi_k \|s_k\|
    \leq
    p_k(s_k) + \| \nabla f(x_{k,+}) -\nabla f(x_{k})- H_ks_k\|\|s_k\|.
\end{equation*}
Furthermore as $f \in C^2$ and $p_k(s_k) \geq \frac{\mu}{2}\|s_k\|^2$ we may enlarge $K_0$ such that additionally 
\begin{equation*}
    p_k(s_k)+ (m_k(s_k) - f(x_k + s_k)) >0 \qquad \text{for $k\geq K_0$}.
\end{equation*}
Consequently we obtain for those $k$
\begin{equation*}
    \rho_k^\vartheta = \frac{f(x_k)-f(x_k+s_k)}{p_k(s_k) + \frac{\vartheta}{2}\psi_k \|s_k\|}
   {\geq}
    \frac{
        p_k(s_k)+ (m_k(s_k) - f(x_k + s_k))
    }{
        \left(1+\frac{\vartheta}{2}\right)p_k(s_k)
        +
        \frac{\vartheta }{2}  \| \nabla f(x_{k,+}) -\nabla f(x_{k})- H_ks_k\|\|s_k\|
    }.
\end{equation*}
If now $p_k(s_k)=0$ then due to \eqref{eq:cat_local_1} also $s_k=0$ and consequently $\|s_k\| <\Delta_k$ and by the complementarity condition, $\lambda_k=0$. This in turn yields $0= H_ks_k =-g_k$, so $x_k$ stationary and this case was excluded by the standing convention in  \Cref{par:standing_convention}. Hence we assume $p_k(s_k) > 0$ and divide in the above inequality numerator and denominator by $p_k(s_k)$ and obtain
\begin{equation}
    \rho_k^\vartheta \geq \frac{
        1+\frac{(m_k(s_k) - f(x_k + s_k))}{p_k(s_k)}
    }{
        \left(1+\frac{\vartheta}{2}\right)
        +
        \frac{\vartheta }{2}  \frac{\| \nabla f(x_{k,+}) -\nabla f(x_{k})- H_ks_k\|}{p_k(s_k)}\|s_k\| \label{eq:cat_local_2}
    }
\end{equation}
As in the proof of \Cref{thm:tr_local_fast_convergence},  since $x_k\to x_\star$ and $\|s_k\|\to 0$, after possibly
enlarging $K_0$, we may fix an $r>0$ and assume that 
\[
    x_k+t s_k\in\overline{B_r(x_\star)}
    \qquad
    \text{for all }t\in[0,1]\text{ and }k\geq K_0.
\]
As $H$ is finite-dimensional and
$\nabla^2f$ is continuous, it is uniformly continuous on
$\overline {B_r(x_\star)}$. Consequently, standard estimates using the integral Taylor formulas yield
\begin{equation*}
m_k(s_k)-f(x_k+s_k)=o(\|s_k\|^2),
\qquad
\|\nabla f(x_{k,+})-\nabla f(x_k)-H_ks_k\|
=o(\|s_k\|).
\end{equation*}
as $\|s_k\| \to 0$. Using now \eqref{eq:cat_local_1}, \eqref{eq:cat_local_11}, it is easy to deduce from \eqref{eq:cat_local_2} and the assumptions of the theorem that
\[
    \liminf_{k\to\infty}\rho_k^\vartheta
    \geq
    \frac{1}{1+\vartheta/2}.
\]
Because $\eta_2<\frac{1}{1+\vartheta/2}$, after possibly enlarging $K_0$ we may assume that
\[
    \rho_k^\vartheta\geq \eta_2
    \qquad
    \forall k\ge K_0 .
\]
Consequently by the radius update rule $\Delta_{k} \geq \underline{\Delta} := \Delta_{K_0}$ for all $k \geq K_0$. By \eqref{eq:cat_local_11} we have $\|s_k\| \to 0$, hence, by possibly increasing $K_0$ again, we may also assume that $\|s_k\| < \Delta_k$ for all $k \geq K_0$, which implies using the complementarity condition, that $\lambda_k=0$.  Hence, for those $k \geq K_0$  we have 
\begin{equation*}
    s_k =  - H_k^{-1}g_k,
\end{equation*}
which is the standard Newton step. Here the inverse exists by assumption. The rest of the proof of $(iv)$--$(v)$ follows standard arguments for Newton methods and can be deduced almost identically as in the proof of \Cref{thm:tr_local_fast_convergence} for $\bar \tau = 0$. We do not repeat the argument here.
\end{proof}
\begin{remark}
As in the previous section, the statement in $(iv)$ remains valid for merely $C^2$--functions $f$ and does not rely on the Hölder condition with exponent $\nu=0$. The proof remains the same.
\end{remark}

\section{Conclusion and outlook}

In this paper, we have shown that a substantial form of universality is already inherent in simple trust-region mechanisms. A variant of the basic inexact trust-region method, equipped with the classical acceptance ratio and radius-update rule, attains the universal convex complexity bound under H\"older continuity of the Hessian. A simple modification of the acceptance ratio, in the spirit of Hamad and Hinder~\cite{hamad2024simple}, further yields the optimal H\"older-adaptive first-order complexity in the nonconvex regime, while preserving both the universal convex complexity and the local fast convergence of Newton's method. Thus, the same quadratic trust-region framework adapts to the nonconvex, convex, and locally strongly convex regimes without requiring knowledge of the H\"older exponent or the corresponding regularity constant. A central ingredient of the convex analysis is the function-gap model-decrease estimate developed in this paper, which appears to have remained largely hidden in the classical trust-region analysis.

Several directions for future research arise from the present results. First, it would be natural to combine the H\"older-continuous analysis developed here with the inexactness framework of~\cite{hamad2024simple}. This would remove the exact-subproblem assumption from the analysis of the UniCAT variant and substantially improve its practical relevance. Similarly, the basic inexact trust-region method relies on an abstract relative model-decrease condition. Although this condition can be satisfied by nearly exact subproblem solvers, such as Mor\'e--Sorensen-type methods, these methods are often not well suited to large-scale or high-dimensional problems. It would therefore be valuable to determine whether the present universal convex analysis can be extended to more scalable subproblem procedures, including Krylov-subspace and truncated conjugate-gradient methods under weaker and more readily verifiable decrease conditions.

A further direction is to investigate whether the same universal mechanism persists in broader trust-region frameworks. Natural candidates include stochastic trust-region methods~\cite{blanchet2019convergence,jin2024high}, proximal trust-region methods~\cite{baraldi2023proximal}, and multilevel variants~\cite{gratton2008recursive,baraldi2025multilevel}. Such extensions would be particularly relevant for large-scale, high-dimensional, and nonsmooth optimization problems. 

Finally, the present work is theoretical. A computational study of the proposed variants, particularly in combination with scalable inexact subproblem solvers, would help assess whether the universal behavior identified by the analysis is also reflected in practice. Relevant benchmark classes include problems arising in machine learning, inverse problems, and engineering design.

\section*{Acknowledgements}

The author is grateful to Amal Alphonse for proofreading substantial portions of the manuscript, and to Robert Baraldi and Coralia Cartis for their valuable literature suggestions.
\printbibliography

\end{document}